%% file: tetra7d.tex
\documentclass[12pt,a4paper]{article}
\usepackage{epsfig}
\usepackage{color} % vkladanie PostScript-ovych obrazkov
\usepackage{latexsym}
\usepackage{amsfonts}
\usepackage{amsmath, amssymb, amsthm}
\usepackage{hyperref}
\usepackage{graphicx}
\usepackage{multicol}
\usepackage{subcaption}

\newtheorem{theorem}{Theorem}[section]
\newtheorem{proposition}[theorem]{Proposition}
\newtheorem{lemma}[theorem]{Lemma}
\newtheorem{corollary}[theorem]{Corollary}
\newtheorem{example}[theorem]{Example}
\newtheorem{remark}[theorem]{Remark}
\newtheorem{definition}[theorem]{Definition}

\newtheorem{problem}[theorem]{Problem}

\newtheorem*{theorem1*}{Theorem}
\newtheorem*{problem1*}{Problem}
\newtheorem*{conjecture1*}{Conjecture 1}
\newcommand{\inv}{^{-1}}

\newcommand{\oto}{\longleftrightarrow}

\newcommand{\dpt}{\mathtt{dpt}}
\newcommand{\dc}{\mathtt{dc}}
\newcommand{\dm}{\mathtt{dm}}
\newcommand{\ang}{\mathtt{ang}}
\newcommand{\ls}{\mathtt{ls}}
\newcommand{\hl}{\mathtt{hl}}
\newcommand{\alt}{\mathtt{alt}}
\newcommand{\ax}{\mathtt{ax}}

\newcommand{\T}{\mathbf{T}}

\newcommand{\Sh}{\boldsymbol{\Sigma}}

\begin{document}

\title{\large{\textbf{
CUBIC GRAPHS THAT CANNOT BE COVERED\\ WITH FOUR PERFECT
MATCHINGS}}}

\author{
Edita M\'a\v cajov\' a and Martin \v{S}koviera\\[3mm]
Department of Computer Science\\
Faculty of Mathematics, Physics and Informatics\\
Comenius University\\
842 48 Bratislava, Slovakia\\[2mm]
{\small\tt macajova@dcs.fmph.uniba.sk}\\[-1mm]
{\small\tt skoviera@dcs.fmph.uniba.sk}}

\date{\today}

\maketitle

\begin{abstract}
A conjecture of Berge suggests that every bridgeless cubic
graph can have its edges covered with at most five perfect
matchings. Since three perfect matchings suffice only when the
graph in question is $3$-edge-colourable, the rest of cubic
graphs falls into two classes: those that can be covered with
four perfect matchings, and those that need at least five.
Cubic graphs that require more than four perfect matchings to cover their
edges are particularly interesting as potential counterexamples
to several profound and long-standing conjectures including the
celebrated cycle double cover conjecture. However, so far they
have been extremely difficult to find.

In this paper we build a theory that describes coverings with
four perfect match\-ings as flows whose flow values and outflow patterns form a
configuration of six lines spanned by four points of the
3-dimensional projective space $\mathbb{P}_3(\mathbb{F}_2)$ in
general position. This theory provides powerful tools for
investigation of graphs that do not admit such a cover and
offers a great variety of methods for their construction.  
As an illustrative example we produce a rich family of
snarks (nontrivial cubic graphs with no $3$-edge-colouring)
that cannot be covered with four perfect matchings. The family
contains all previously known graphs with this property.

\bigskip\noindent\textbf{Keywords:}
cubic graph, snark, perfect matching, covering, cycle double cover conjecture

\bigskip\noindent\textbf{AMS subject classifications:}
05C21, 05C70, 05C15.
\end{abstract}

\section{Introduction}

Ever since Petersen proved his Perfect Matching Theorem
\cite{Pet}, perfect matchings in cubic graphs have been
regarded as their fundamental structures. An extension of
Petersen's theorem due to Sch\"onberger \cite{Sch}, later
generalised by Plesn\'\i k \cite{Plesnik}, implies that every
bridgeless cubic graph $G$ contains a set of perfect matchings
that together cover all the edges of $G$ (see also
\cite[p.~192]{Konig} and \cite[Corollary~3.4.3]{LP}). A natural
question arises as to what is the minimum number of perfect
matchings needed to cover all edges of a given bridgeless cubic
graph $G$; this number is called the \textit{perfect matching
index} of~$G$ and is denoted by $\pi(G)$ (an alternative term
\textit {excessive index} also occurs, see \cite{AKLM, EM}.)

\begin{figure}[ht]
\begin{subfigure}{.5\textwidth}
  \centering
 \includegraphics[width=.7\linewidth]{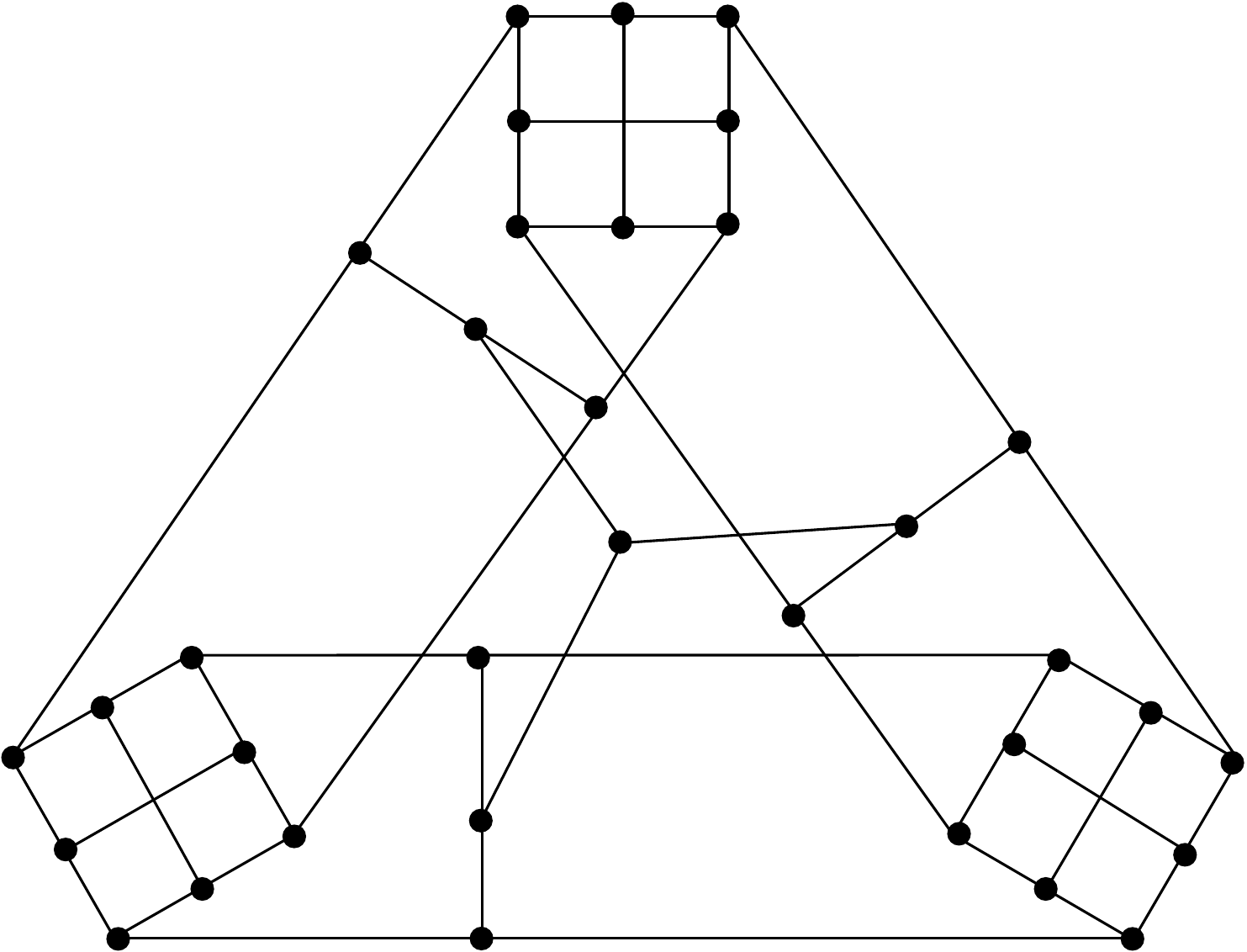}
\end{subfigure}
\begin{subfigure}{.5\textwidth}
  \centering
  \includegraphics[width=.65\linewidth]{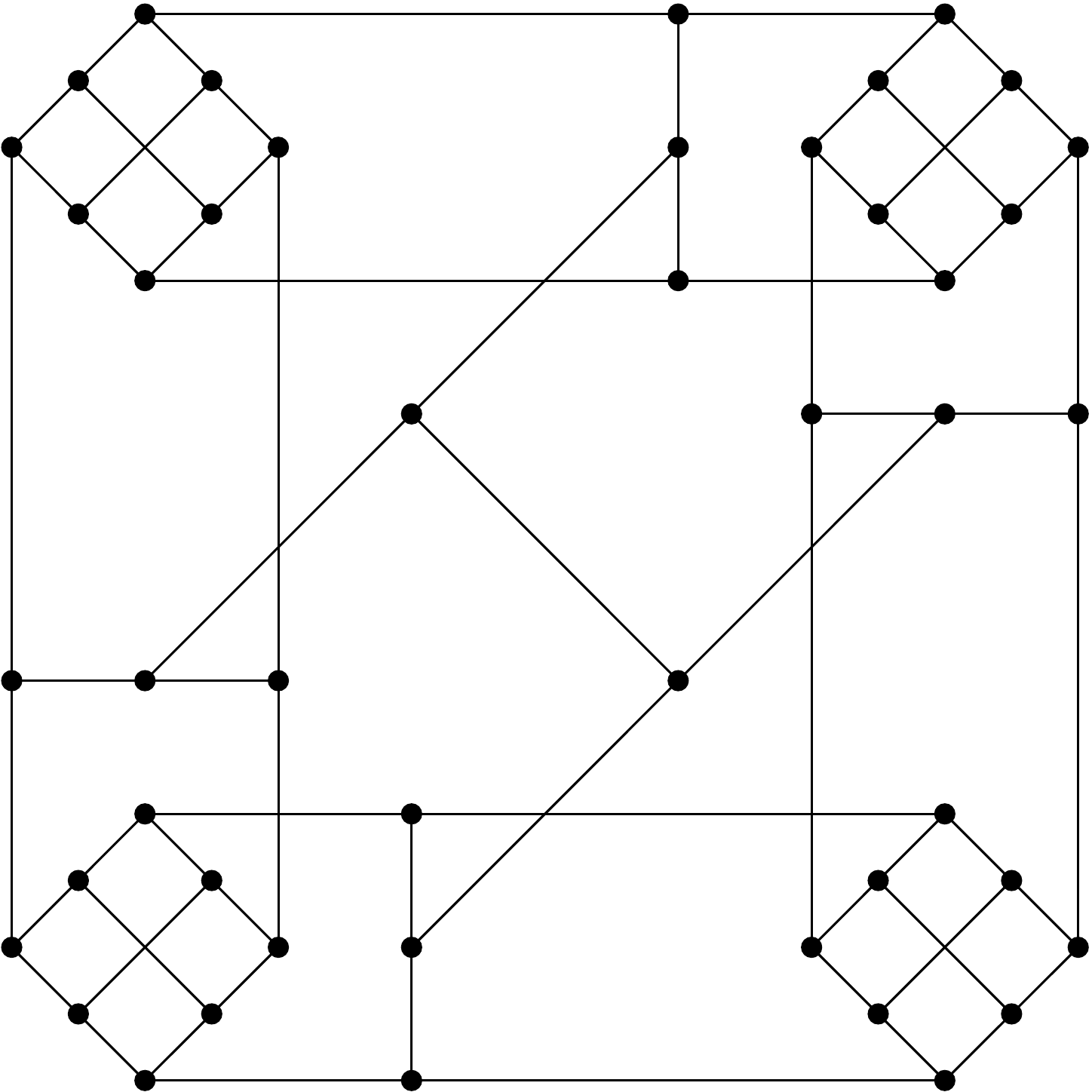}
\end{subfigure}
\caption{Graphs of order 34 and 46 that cannot be covered with
         four perfect matchings}
\label{fig:g34+46}
\end{figure}

Clearly, $\pi(G)\ge 3$ for every bridgeless cubic graph $G$
with equality achieved if and only if the graph is
$3$-edge-colourable. On the other hand, no constant upper bound
for the perfect matching index of a cubic graphs is known.
However, a conjecture attributed to Berge~(see \cite{M1} or
\cite{S}) suggests that $\pi(G)\le 5$ for every bridgeless
cubic graph~$G$. If this conjecture is true, then the perfect
matching index of every \textit{snark}, a 2-connected cubic
graph with no $3$-edge-colouring, is to be either $4$ or $5$.

Cubic graphs with $\pi=4$ enjoy several important properties
well known to hold for $3$-edge-colourable graphs. For example,
every cubic graph with $\pi=4$ satisfies the $5$-cycle double
conjecture (see \cite[Theorem~3.1]{HLZ} and 
\cite[Theorem~3.1~(2)]{Steffen}), the $7/5$-conjecture of Alon
and Tarsi and Jaeger \cite[Theorem~3.1~(1)]{Steffen} on
shortest cycle covers, the Fan-Raspaud \cite{FR} conjecture on three
perfect matchings with empty intersection, and others.
Therefore cubic graphs with $\pi\ge 5$ are of particular
interest, providing potential counterexamples to these and
several other related conjectures such as the Fulkerson
conjecture.

Very little is known about cubic graphs with $\pi\ge 5$. One
major difficulty in their study comes from the fact that they
appear to be extremely rare and therefore hard to find. The
smallest such graph is the Petersen graph. Arbitrarily large
examples with connectivity $2$ can be easily derived from the
Petersen graph, so the real problem is constructing nontrivial
examples different from the Petersen graph. In 2009, Fouquet
and Vanherpe \cite[Problem~4.3]{FH} asked whether there exists
a cyclically $4$-edge-connected cubic graph with $\pi\ge 5$
different from the Petersen graph. The first such graph was
reported by Brinkmann et al. \cite{BGHM} in 2013 as a result of
an exhausting computer search. In the list comprising all
64~326~024 cyclically $4$-edge-connected cubic snarks of girth
at least $5$ with $\pi\ge 4$ on up to 36 vertices there are
only two that cannot be covered with four perfect matchings:
the Petersen graph and the graph on 34 vertices displayed in
Figure~\ref{fig:g34+46} (left). Another sporadic example, on 46
vertices, was discovered by H\"agglund
\cite[Section~3]{Hagglund}; it is depicted on
Figure~\ref{fig:g34+46} (right). Esperet and Mazzuoccolo
\cite{EM} generalised the former graph to an infinite family of
\textit{windmill snarks} with $\pi\ge 5$. A similar infinite
family, which takes the other graph from
Figure~\ref{fig:g34+46} as its basis, was provided by Chen
\cite{Chen}. Both graphs from Figure~\ref{fig:g34+46} are
included in another infinite family of graphs with $\pi\ge 5$,
the family of \textit{treelike snarks} constructed by Abreu et
al. \cite{AKLM}, which we now discuss in a greater detail.

In some sense, treelike snarks are much richer in form than
the windmill snarks and the snarks of Chen. Their
construction makes use of cubic Halin graphs. A \textit{Halin
graph} $H$ consists of a plane tree $S$ with no $2$-valent
vertices and a circuit $C$ passing through all the leaves of
$S$ in such a way that $H=C\cup S$ remains embedded in the
plane. It follows that the circuit $C$, called the
\textit{perimeter circuit}, forms the boundary of the outer
face of~$H$, while the tree $S$, the \textit{inscribed tree} of
$H$, is contained in the interior of $C$. Halin \cite{Halin}
introduced these graphs in 1971 as a class of minimally
$3$-connected graphs. However, the cubic Halin graphs were
studied much earlier by Rademacher \cite{Rademacher} in 1965
and by Kirkman \cite{Kirkman} as early as in 1856.

A treelike snark is formed from a cubic Halin graph $H$ by
substituting each vertex of the perimeter circuit $C$ of $H$
with a 11-vertex subgraph $F_{Ps}$ found in both graphs from
Figure~\ref{fig:g34+46}, called the \textit{Petersen fragment}
of \cite{AKLM}. This makes a significant difference in
comparison with the windmill snarks and the family of Chen
where more general building blocks can be used.

One of the notable aspects of the study of cubic graphs with
$\pi\ge 5$ is the increasing difficulty of proving that four
perfect matchings are not enough to cover their edges. This
problem becomes evident if we compare the proof of Esperet and
Mazzuoccolo \cite{EM} for the windmill snarks with that for the
family constructed by Chen \cite{Chen}, and even more so if we
take into account the proof for treelike snarks \cite{AKLM},
where a substatial computer assistance is required. The reason
for this lies in the lack of a simple characterisation of
graphs with perfect matching index $5$ or more. This fact was
realised already by H\"agglund \cite{Hagglund} who posed the
following problem.

\begin{problem1*}[H\"agglund \cite{Hagglund}, Problem~3]
{\rm Is it possible to give a simple characterisation of cubic
graphs with perfect matching index equal to 5?}
\end{problem1*}

In response to H\"agglund's question we provide a
characterisation of graphs with perfect matching index not
exceeding $4$. We describe them as the graphs that admit a
proper edge-colouring by points of the configuration $T$ of ten points and six
lines in the $3$-dimensional projective space
$\mathbb{P}_3(\mathbb{F}_2)=PG(3,2)$, over the $2$-element
field, spanned by four points in general position (see
Figure~\ref{fig:tetra}). The defining property of the colouring
requires any three edges incident with the same vertex to carry
colours that form a line of $T$. Although this result seems to
be nothing more than a simple observation, it provides a
surprisingly powerful tool for the study of both cubic graphs
that can, as well as those that cannot, be covered with four
perfect matchings. Especially useful is the property that the such
colourings are in fact nowhere-zero flows with values in the
group $\mathbb{Z}_2^4$.

\begin{figure}[htbp]
\begin{center}
     \scalebox{0.45}
     {
       \input{pp}
     }
\end{center}
\caption{A tetrahedron in $PG(3,2)$ spanned by points
$p_1$, $p_2$, $p_3$, and $p_4$}\label{fig:tetra}
\end{figure}
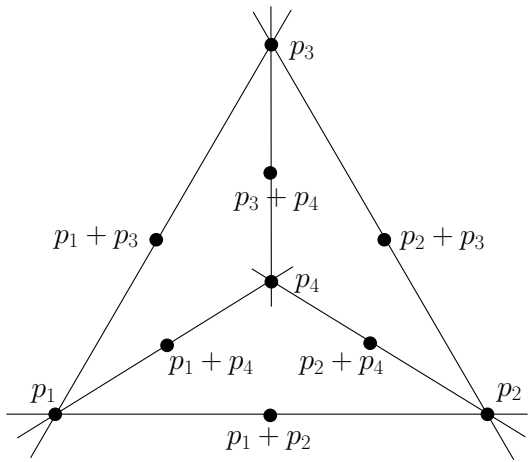

We demonstrate the power of our approach by presenting a
far-reaching generalisation of treelike snarks. The new family
of graphs, which we call \textit{Halin snarks}, contains all
previously known nontrivial examples of graphs with $\pi\ge 5$.
In particular,  the treelike snarks, the windmill snarks, and
the snarks constructed by Chen \cite{Chen} are all Halin
snarks.

\begin{figure}[htbp]
\begin{center}
     \scalebox{0.45}
     {
       \input{Halin}
     }
\end{center}
\caption{A Halin snark}\label{fig:halin}
\end{figure}
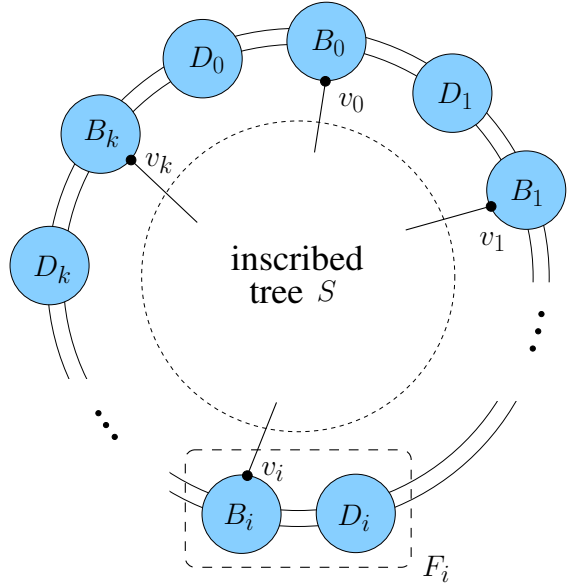

The construction of a Halin snark is similar to that of a
treelike snark. It starts with a cubic Halin graph $H=C\cup S$
where $C$ is the perimeter circuit and $S$ is the inscribed
tree of $H$. Each vertex of $C$ is substituted by a
\textit{Halin fragment} $F$, a generalisation of the Petersen
fragment $F_{Ps}$ mentioned above. It is formed from the union
of two graphs $D$ and~$B$: the graph $D$ is obtained from an
arbitrary bridgeless cubic graph $G$ with $\pi(G)\ge 5$ by
removing two adjacent vertices $u$ and $v$ and retaining the
four dangling edges; the graph $B$ is obtained from an
arbitrary cubic bipartite graph $G'$ by removing a path
$u'w'v'$ of length $2$ and retaining the five dangling edges.
The dangling edges of $D$ formerly incident in $G$ with $v$ are
then welded with those of $B$ formerly incident in $G'$ with
$u'$, thereby producing $F$.

Each Halin fragment $F$ has five dangling edges arranged in two
pairs, which correspond to the vertices $u$ and $v'$, and one
singleton, which corresponds to the vertex $w'$. The
construction of a Halin snark is now finished by substituting
each vertex $v_i$ of the perimeter circuit $C$ of $H$ with a
Halin fragment $F_i$ in such a way that the lonely dangling
edge of $F_i$ will replace the edge of $S$ incident with $V_i$,
and each of pairs of dangling edges of $F_i$ will replace one
edge of $C$ incident with $v_i$ following a cyclic orientation of $C$. Finally, 
the pairs of dangling
edges from Halin fragments $F_i$ and $F_j$ corresponding to
adjacent vertices $v_i$ and $v_j$ of $C$ are welded in such a way
that a cubic graph is obtained (see Figure~\ref{fig:halin}). We denote the 
resulting graph
by $H^{\sharp}=C^{\sharp}\cup S$ where $S$ is the inherited
inscribed tree of $H$ and $C^{\sharp}$ is the graph obtained
from the perimeter circuit by performing all required
substitutions.

If each Halin fragment $F_i$ used for the construction of
$H^{\sharp}$ is created from a cubic graph $G_i$ with
$\pi(G)\ge 5$ and from a bipartite cubic graph $G_i'$, both
cyclically $4$-edge-connected and of girth at least $5$, then
the same holds for the resulting graph $H^{\sharp}$.

The proof that every Halin snark has perfect matching index at
least $5$ relies on simple geometric considerations involving
the projective space $PG(3,2)$ combined with nowhere-zero flow
arguments. It is completely computer-free.

The methods developed along the way permit further extensions
of the family of Halin snarks, suggesting that the class of all
cubic graphs with perfect matching index at least $5$ may have
a very complicated structure. In particular, the following
theorem is a consequence of our results.

\begin{theorem1*}\label{thm:intro1}
For every even integer $n\ge 42$ there exists a cyclically
$4$-edge-conncected cubic graph $G$ of girth at least $5$ on
$n$ vertices such that $\pi(G)\ge 5$.
\end{theorem1*}

Taking into account the exhaustive computer search performed by
Brinkmann et al. \cite{BGHM} mentioned earlier,
Theorem~\ref{thm:intro1} leaves the existence of a cyclically
$4$-edge-connected cubic graph $G$  of girth at least $5$ on
$n$ vertices with $\pi(G)\ge 5$ open only for $n=38$ and
$n=40$. This is a significant progress as all previously known
constructions combined are capable of producing only graphs of
order $n=12m-2$, where $m\ge 3$ (see \cite{AKLM, Chen, EM}).

This paper is organised as follows. The next section contains a
brief survey of terminology and notation used later in the
paper. Section~\ref{sec:char} presents a characterisation of
cubic graphs that cannot be covered with four perfect matching
in terms of flows with additional geometric structure within
the projective space $PG(3,2)$. Geometric language is further
developed in Section~\ref{sec:objects} with focus on objects
consisting of two or three points. Transformation of geometric
objects via the flows corresponding to coverings with four
perfect matchings is studied in the next three sections. In
Section~\ref{sec:Halin}, the accumulated knowledge is applied to
proving that all Halin snarks have perfect matching index at
least $5$ and that they are nontrivial snarks whenever the
building blocks are taken from cyclically $4$-edge-connected
graphs with girth at least $5$. The next section deals with
circular flow number of Halin snarks, which is shown to be at
least $5$ whenever the building have this property. A few final
remarks conclude the paper.

The present paper serves as an introduction to the topic and
its results will be extensively used in our subsequent papers.

\section{Preliminaries}\label{sec:prelim}

All graphs in this paper are finite and for the most part
simple and cubic (3-valent). Multiple edges and loops may
occur, but they will usually be excluded by imposing additional
restrictions.

Graphs can be assembled from smaller building blocks called
multipoles. Similarly to graphs, each \textit{multipole} $M$
has its vertex set $V(M)$, its edge set $E(M)$, and an
incidence relation between vertices and edges. Each edge of $M$
has two ends, and each end may, but need not be, incident with
a vertex of $M$. An edge that is not incident with a vertex is
called a \emph{dangling edge}; its free end is called a
\textit{semiedge}. A multipole with $k$ semiedges is called a
\textit{$k$-pole}. Any two edges $s$ and $t$, each with a free
end, can be coalesced into a new edge $s*t$, the
\textit{junction} of $s$ and $t$, by identifying a free end of
$s$ with a free end of $t$; the incidences of $s*t$ are
naturally inherited from $s$ and $t$.

Semiedges in multipoles are often grouped into pairwise
disjoint sets, called \textit{connectors}. A semiedge not
occurring in any connector is said to be \textit{residual}.  In
this paper, most multipoles will be dipoles. A \textit{dipole}
is a multipole with two connectors referred to as the
\textit{input connector} and the \textit{output connector}. A
dipole $D(I,O)$ with input connector $I$ of size $a$, output
connector $O$ of size $b$, and with $c$ residual semiedges,
will be called an \textit{$(a,b;c)$-pole}. If $c=0$, we speak
of an $(a,b)$-pole. For an $(a,a)$-pole $D(I,O)$ we define its
\textit{closure} $[D]$ to be a graph obtained from $D$ by
ordering the semiedges in both connectors and performing the
junction of the $i$-th semiedge of the output connector to the
$i$-th semiedge of the input connector.

A common way of constructing multipoles is by removing vertices
from a graph. The following multipoles obtained from a cubic
graph $G$ will be repeatedly used throughout the paper: Let
$v$, $uv$, and $uwv$ be paths in $G$ of length 0,~1, and 2, respectively. 
Denote by $G_v$, $G_{uv}$, and $G_{uwv}$, respectively,
the $3$-pole, the $(2,2)$-pole, and the $(2,2;1)$-pole obtained
from $G$ by removing the corresponding path and arranging the
resulting semiedges in such a way that the edges formerly
incident with the same vertex belong to the same connector. To
be more precise, in both $G_{uv}$ and $G_{uwv}$ the semiedges
formerly incident with $u$ will be assigned to the input
connector, those formerly incident with $v$ will be assigned to
the output connector, and the semiedge incident with $w$ will
be residual. The semiedges of $G_v$ will all belong to its
single connector.

An \textit{edge colouring} of a graph or a multipole $X$ is an
assignment of colours from a set $Z$ of \textit{colours} to the
edges of $X$ in such a way that the edges with adjacent edge
ends receive distinct colours. It means that all edge
colourings in this paper are proper. A
\textit{$k$-edge-colouring} is one where $|Z|=k$ colours. A
$3$-edge-colouring is also known as a \textit{Tait colouring}.
A $2$-connected cubic graph with no $3$-edge-colouring is
called a \textit{snark}.  A snark is \textit{nontrivial} if it
is cyclically $4$-edge-connected and has girth at least $5$.

A natural generalisation of a Tait colouring is the concept of
a \textit{local Tait colouring} of a cubic graph proposed by
Archdeacon \cite{A} and developed in \cite{HS, KMS, MS-Fano},
and elsewhere. It allows for an unlimited number of colours but
requires that the colours of any two edges meeting at a vertex
always determine the same third colour. Local Tait colourings
can be conveniently described in terms of colourings by points
of a partial Steiner triple system such that the colours
meeting at any vertex form a triple of the system; for details
see, for example, \cite{KMS}.

Local Tait colourings that occur in the present paper are at
the same time nowhere-zero flows. The pertinent definitions are
therefore in order. Given an abelian group $A$, an
\textit{$A$-flow} on a graph $G$ consists of an orientation of
$G$ and a function $\phi\colon E(G)\to A$ such that, at each
vertex, the sum of all incoming values equals the sum of all
outgoing ones (\textit{Kirchhoff's law}). A flow which only
uses non-zero elements of the group is said to be
\textit{nowhere-zero}. For the existence of an $A$-flow on $G$
the choice of the edge directions is immaterial for one can
reverse the orientation of any edge and replace the value on it
by its negative without violating the Kirchhoff law.
Furthermore, if $x=-x$ for every $x\in A$, then orientation can
be ignored altogether. This is the case of local Tait
colourings encountered in this paper.

\section{Covering cubic graphs with four perfect
matchings}\label{sec:char}

The aim of this section is to translate the problem of covering
a cubic graph with four perfect matchings into the language of
flows whose values and flow patterns around vertices are
restricted to points and lines of a tetrahedron in the
$3$-dimensional projective space over the $2$-element field
$\mathbb{F}_2$. This new language will be crucial for the
remainder of the paper. We start with the necessary geometric
background.

The \textit{$n$-dimensional projective space}
$\mathbb{P}_n(\mathbb{F}_2)=PG(n,2)$ over the $2$-element field
$\mathbb{F}_2$ is an incidence geometry whose \textit{points}
can be identified with the nonzero vectors of the
$(n+1)$-dimensional vector space $\mathbb{F}_2^{n+1}$ and
\textit{lines} are formed by the triples $\{x,y,z\}$ of points
such that $x+y+z=0$. The
$2$-dimensional projective space $PG(2,2)$ is the well-known
Fano plane which has $7$ points and $7$ lines. The
$3$-dimensional projective space $PG(3,2)$ has $15$ points and
$35$ lines.

An automorphism of $PG(n,2)$ is called a \textit{collineation}.
It maps lines to lines and hence collinear points to collinear
points. It is well known \cite{Coxeter} that each collineation
of $PG(n,2)$ is induced by a bijective linear transformation
$\mathbb{F}_2^{n+1}\to \mathbb{F}_2^{n+1}$.

A \textit{tetrahedron} in $PG(3,2)$ is a configuration $T$
consisting of ten points and six lines spanned by a set of four
points of $PG(3,2)$ in general position; it means that no three
of the points lie on the same line. Two points of $T$ that lie
on the same line of $T$ are said to be \textit{collinear}
in~$T$, otherwise they are \textit{non-collinear} in $T$.

Consider a fixed tetrahedron $T=T(p_1,p_2,p_3,p_3)$ spanned by
points $p_1$, $p_2$, $p_3$, and $p_4$ in general position;
these four points are the \textit{corner points} of $T$. With
every pair $\{c_1,c_2\}$ of distinct corner points $T$ contains
the entire line $\{c_1,c_1+c_2, c_2\}$. The point $c_1+c_2$ is
the \textit{midpoint} of the line $\{c_1,c_1+c_2, c_2\}$.
Clearly, there are six midpoints in $T$. Thus, in total, $T$
has four corner points and six midpoints arranged in six lines
(see Figure~\ref{fig:tetra}). Observe that each line
$\ell=\{c_1,c_1+c_2, c_2\}$ of $T$ is uniquely determined by
either its two corner points or by its midpoint; accordingly,
$\ell$ will be denoted by $\langle c_1,c_2\rangle$ or by
$\langle c_1+c_2\rangle$.

There are exactly five points of $PG(3,2)$ that are not
included in $T$. In order to describe them in geometric terms
recall that four points $x_1$, $x_2$, $x_3$, and $x_4$ of
$PG(3,2)$ are in general position if and only if they form a
basis of the vector space $\mathbb{F}_2^4$. Every vector
$y\in\mathbb{F}_2^4$ can therefore be uniquely expressed as a
linear combination
$$y=\alpha_1p_1+\alpha_2p_2+\alpha_3p_3+\alpha_4p_4$$
where the coefficients $\alpha_i$ are from
$\mathbb{F}_2=\{0,1\}$. The number of nonzero coefficients in
this expression is the \textit{weight} of $y$ with respect to
$T$, and will be denoted by $|y|_T$. The subscript $T$ will be
dropped whenever $T$ is clear from the context. We emphasise
that the weight $|y|_T$ coincides with the Hamming weight only
when $T$ is spanned by the unit vectors of~$\mathbb{F}_2$.

According to our definition of weight, the zero vector has
weight $0$, the corner points of $T$ have weight $1$, and
midpoints have weight $2$. The remaining five points of
$PG(3,2)$ have weight $3$ and $4$ and can be characterised as
follows. In $T$, consider a \textit{triangle} $t$, by which we
mean a set consisting of six points arranged in three lines
spanned by three distinct corner points $c_1$, $c_2$, and
$c_3$; we denote the triangle $t$ by $\langle
c_1,c_2,c_3\rangle$. The point $c_1+c_2+c_3$ of $PG(3,2)$,
which obviously does not belong to $T$, can be regarded as the
\textit{centre} of $t$. There are four triangles in $T$; their
centres provide four of the five points of $PG(3,2)$ missing
in~$T$. The last missing point is $p_1+p_2+p_3+p_4$, the
\textit{barycentre} of the entire~$T$.

Let us consider an arbitrary covering $\mathcal{M}$ of a cubic
graph $G$ with four perfect matchings $M_1$, $M_2$, $M_3$,
and~$M_4$, not necessarily distinct; thus $E(G)=M_1\cup M_2\cup
M_3\cup M_4$. Every vertex $v$ of $G$ is incident with all the
members of $\mathcal{M}$ and each edge incident with $v$
belongs to a member of $\mathcal{M}$. It follows that one of
the edges at $v$ is covered by two members of $\mathcal{M}$
while the remaining two are covered by a single member of
$\mathcal{M}$. If we label each edge $e$ with the binary vector
$(x_1,x_2,x_3,x_4)$ where $x_i=0$ if and only if $e$ belongs to
$M_i$ we obtain a mapping
\begin{equation}\label{eq:cover2flow}
\psi=\psi_{\mathcal{M}}\colon E(G)\to \mathbb{Z}_2^4.
\end{equation}
This mapping is easily seen to be a proper edge-colouring of
$G$, in fact, a local Tait colouring in the sense of
\cite{KMS}. Even more, it is a nowhere-zero
$\mathbb{Z}_2^4$-flow because for each $i\in\{1,2,3,4\}$ the
$i$-th coordinate mapping $$e\mapsto\psi(e)_i$$ coincides with
the characteristic function of a cycle, the $2$-factor
complementary to $M_i$, and hence is a flow.

The flow $\psi_{\mathcal{M}}$ has an additional geometric
structure evinced by the fact that the set of values of $\phi$
forms the tetrahedron spanned by the points $(0,1,1,1)$,
$(1,0,1,1)$, $(1,1,0,1)$, and $(1,1,1,0)$ (representing the
perfect matchings $M_1$, $M_2$, $M_3$, and~$M_4$, respectively)
and the flow values around each vertex form a line of the
tetrahedron. Such a flow will be called \textit{tetrahedral}.
To be more precise, let $\phi\colon E(G)\to \mathbb{Z}_2^4$ an
edge-valuation of a cubic graph $G$ and let
$T=T(p_1,p_2,p_3,p_4)$ be a tetrahedron in $PG(3,2)$. We say
that $\phi$ is a \textit{$T$-flow} if, for each edge $e$ of
$G$, the value $\phi(e)$ is a point of $T$, and for each vertex
$v$ of $G$, the set $\{\phi(e_1),\phi(e_2),\phi(e_3)\}$, where
$e_1$, $e_2$, and $e_3$ are the edges incident with~$v$, is a
line of~$T$. If the particular tetrahedron $T$ is irrelevant,
we just say that $\phi$ is a \textit{tetrahedral flow}.  Note
that any tetrahedral flow is a nowhere-zero flow because any
three points constituting a line of $T$ sum to zero and the
value $\phi(e)=0$ doe not occur. Since $\phi$ is also a proper
edge-colouring, it will sometimes be referred to as a
\textit{tetrahedral colouring} (or a \textit{$T$-colouring}
whenever a tetrahedron $T$ is specific) and the values
$\phi(e)$ as \textit{colours}. It may be worth mentioning that
a tetrahedral flow is also an instance of a $B$-flow in the
sense of Jaeger \cite[p.~73]{Jaeger}, where $B$ is a subset of
an abelian group $A$ not containing zero such that $B=-B$.
However, it is endowed with an additional geometric structure.

\medskip

With this preparation we can proceed to the main
result of this section.

\begin{theorem}\label{thm:1-1}
A cubic graph can have its edges covered with four perfect
matchings if and only if it admits a tetrahedral flow.
Moreover, there exists a one-to-one correspondence between
coverings of $G$ with four perfect matchings and $T$-flows,
where $T$ is an arbitrary fixed tetrahedron in $PG(3,2)$.
\end{theorem}

\begin{proof}
It suffices to prove the second statement. We first do it for
the tetrahedron $T_1$ spanned by the points $(0,1,1,1)$,
$(1,0,1,1)$, $(1,1,0,1)$, and $(1,1,1,0)$, and then supply a
general argument.

As mentioned in the preceding paragraphs, given a covering
$\mathcal{M}=\{M_1,M_2,M_3,M_4\}$ of $G$ with four perfect
matchings, the mapping $\psi_{\mathcal{M}}$ defined by
(\ref{eq:cover2flow}) is a $T_1$-flow. For the converse, let
$\phi$ be an arbitrary $T_1$-flow on $G$. For each
$i\in\{1,2,3,4\}$ define $N_i$ to be the set of all edges where
the $i$-coordinate of the point $\psi(e)$ equals $0$. Taking
into account the structure of lines in $T_1$ it is not
difficult to see that each $N_i$ is a perfect matching of $G$
and that $\mathcal{N}=\{N_1,N_2,N_3,N_4\}$ covers all the edges
of $G$. Thus every $T_1$-flow $\phi$ on $G$ determines a
covering $\mathcal{N}=\mathcal{N}_{\phi}$ of $G$ with four
perfect matchings. Furthermore, if we start from a $T_1$-flow
$\phi$ on $G$, construct the corresponding covering
$\mathcal{}=\mathcal{N}_{\phi}$, and then derive the $T_1$-flow
$\psi_{\mathcal{N}}$ from it, we can easily check that
$\psi_{\mathcal{N}}=\phi$. Similarly, if we start with a
covering $\mathcal{M}$, derive $\psi=\psi_{\mathcal{M}}$, and
then $\mathcal{N}_{\psi}$, we can conclude that
$\mathcal{N}_{\psi}=\mathcal{M}$. This means that we have
established a one-to-one correspondence between coverings of
$G$ with four perfect matchings and $T_1$-flows.

Now let $T=T(p_1,p_2,p_3,p_4)$ be an arbitrary tetrahedron in
$PG(3,2)$ with corner points $p_1$, $p_2$, $p_3$, and $p_4$.
Since the set $\{p_1,p_2,p_3,p_4\}$ forms a basis of
$\mathbb{F}_2^4$, there exists a linear transformation $\Theta$
of $\mathbb{F}_2^4$ which takes the basis $\{p_1,p_2,p_3,p_4\}$
to the basis $\{(0,1,1,1), (1,0,1,1), \penalty0 (1,1,0,1),
\penalty0 (1,1,1,0)\}$. This mapping induces a collineation of
$PG(3,2)$, which means that $\Theta$ takes a line to a line
and, consequently, transforms $T$ into~$T_1$. In particular, if
$\psi$ is a $T$-flow on $G$, then $\Theta\psi$ is a $T_1$-flow,
and if $\phi$ is a $T_1$-flow, then $\Theta\inv\phi$ is a
$T$-flow. This allows us to conclude that the assignment
$\psi\mapsto\Theta\psi$ establishes a one-to-one correspondence
between $T$-flows and $T_1$-flows on $G$. Combining this
correspondence with the one-to-one correspondence between
$T_1$-flows on $G$ and coverings of $G$ with four perfect
matchings we obtain the desired result.
\end{proof}

The significance of Theorem~\ref{thm:1-1} resides in the fact
that it enables us to move freely between the coverings of
cubic graphs with four perfect matchings and the tetrahedral
flows. With this correspondence in hand, we can replace
reasoning about $1$-factors and $2$-factors in cubic graphs
with algebraic calculus in the group $\mathbb{Z}_2^4$ combined
with the geometry of the projective space $PG(3,2)$ and then
translate the results back to graph structure.

\begin{remark}\label{rem:T0-T1}{\rm
We discuss the relationship between $T_1$-flows and $T$-flows
described in the second part of the previous proof for the
special case of the tetrahedron $T_0$ spanned by the points
$(1,0,0,0)$, $(0,1,0,0)$, $(0,0,1,0)$, and $(0,0,0,1)$. Observe
that both $T_0$ and $T_1$ have the same set of midpoints,
namely the points with exactly two coordinates equal to~$1$.
The linear transformation $\Lambda$ of $\mathbb{F}_2^4$
determined by the matrix
\begin{equation*}
A_{\Lambda} =
\begin{pmatrix}
0 & 1 & 1 & 1\\
1 & 0 & 1 & 1\\
1 & 1 & 0 & 1\\
1 & 1 & 1 & 0
\end{pmatrix}
\end{equation*}
swaps each unit vector of $\mathbb{F}_2^4$ with its
\textit{antipode}, the vector obtained by from it by
interchanging zeros with ones, while leaving the midpoints
fixed; in particular, $\Lambda^2=\mathrm{id}$. The
corresponding collineation maps $T_0$ isomorphically to $T_1$
and vice versa. Now, if $\phi$ is an arbitrary $T_i$-flow on a
cubic graph $G$ for $i\in\{0,1\}$, then the corresponding
$T_{1-i}$-flow $\Lambda\phi$ can be obtained simply by
replacing each flow value of weight $1$ with respect to $T_i$
with its antipode (which is of weight $1$ in the other
tetrahedron) while leaving the flow values of weight $2$
intact.

This correspondence has a useful consequence that if $\phi$ is
an arbitrary $T_0$-flow on a cubic graph $G$, then for each
$i\in\{1,2,3,4\}$ the set $ M_i=\{e\in E(G); \phi(e)_i=1\} $ is
a perfect matching of $G$ and $\mathcal{M}=\{M_1,M_2,M_3,M_4\}$
is a covering of $G$ with four perfect matchings. The
correspondence also works in the reverse direction. \qed
}\end{remark}

\section{Small geometric objects in $PG(3,2)$}
\label{sec:objects}

One way of applying Theorem~\ref{thm:1-1} to proving that a
cubic graph cannot be covered with four perfect matchings is by
analysing conflicting behaviour of tetrahedral flows on the
components resulting from the removal of an edge-cut from the
graph. The sets of flow values on the edges of the cut form
geometric objects in $PG(3,2)$ which in turn can be used to
describing the conflicts. In this section we introduce several
types of objects, mainly of size $2$, that will serve for this
purpose.

We say that two sets $A$ and $B$ of points of a tetrahedron $T$
in $PG(3,2)$ have the \textit{same shape} if there exists a
collineation of $PG(3,2)$ that preserves $T$ and takes $A$ to
$B$. A \textit{geometric shape} in $T$, or simply a
\textit{shape}, is an equivalence class of all point sets
having the same shape. The \textit{shape} of a set of points of
$T$ is a geometric shape it belongs to.

Let $\boldsymbol{\Pi}=\boldsymbol{\Pi}(T)$ be the set of all
pairs $\{p,q\}$ where $p$ and $q$ are points of $T$, not
necessarily distinct. 

We distinguish the following seven types of objects in
$\boldsymbol{\Pi}(T)$:
\begin{itemize}

\item[(i)] A \textit{line segment} is a pair $\{c_1,c_2\}$
    where $c_1$ and $c_2$ are any two distinct corner
    points of $T$. A line segment is a subset of a line
    consisting of its two corner points.  The set of all
    line segments of $T$ will be denoted by $\ls$. Clearly,
    there are six line segments in $T$.

\item[(ii)] A \textit{half-line} is a pair $\{c_1,c_1+c_2\}$
    where $c_1$ and $c_2$ are any two distinct corner
    points of $T$. A half-line is a subset of a line
    consisting of a corner point and a midpoint. It means
    that each line has two half-lines. The point $c_1$ it
    the \textit{origin} of the half-line, and $c_2$ is its
    \textit{target}. The set of all half-lines of $T$ will
    be denoted by $\hl$. There are twelve half-lines in
    $T$.

\item[(iii)] An \textit{angle} is a pair $\{c_1+c_2,
    c_1+c_3\}$ where $c_1$, $c_2$, and $c_3$ are any three
    distinct corner points of $T$. An angle consists of the
    midpoints of two intersecting lines of $T$. Their
    corner points form a triangle. Clearly, each triangle
    $\langle c_1,c_2,c_3\rangle$ in $T$ has three angles.
    On the other hand, each angle belongs to a unique
    triangle. The set of all angles of $T$ will be denoted
    by $\ang$. There are twelve distinct angles in $T$.

\item[(iv)] An \textit{altitude} is a pair
    $\{c_1,c_2+c_3\}$ where $c_1$, $c_2$, and $c_3$ are any
    three distinct corner points of $T$.  An altitude
    consist of a corner point of a triangle $t=\langle
    c_1,c_2,c_3\rangle$ and the midpoint of the line of $t$
    not incident with the chosen the corner point. Each
    triangle has three altitudes. The set of all altitudes
    of $T$ will be denoted by $\alt$. There are twelve
    altitudes in $T$.

\item[(v)] An \textit{axis} is a pair $\{c_1+c_2,c_3+c_4\}$
    where $c_1$, $c_2$, $c_3$, and $c_4$ are all four
    corner points of $T$ in some order. An axis consists of
    the midpoints of two skew (non-intersecting) lines of
    $T$. The set of all axes of $T$ will be denoted by
    $\ax$. Any tetrahedron $T$ in $PG(3,2)$ has three axes.
    The lines of $PG(3,2)$ spanned by the axes of $T$ all
    meet at the barycentre of $T$.

\item[(vi)] A \textit{double corner point} is a degenerate
    pair $\{c,c\}$ where $c$ is an arbitrary corner point
    of $T$. The set of all such pairs will be denoted by
    $\dc$.

\item[(vii)] A \textit{double midpoint} is pair $\{m,m\}$
    where $m$ is an arbitrary midpoint of $T$. The set of
    all such pairs  will be denoted by $\dm$.
\end{itemize}

The pairs under items (i)-(ii) are \textit{collinear}, those
under (iii)-(v) are \textit{non-collinear}, and the pairs under
items (vi)-(vii) are \textit{degenerate}. Clearly, for
a pair $\{p,q\}\in\boldsymbol{\Pi}(T)$ the sum $p+q$ is a point of $T$ if and
only if $\{p,q\}$ is collinear.

\medskip

We now show that the items (i)-(vii) represent all geometric
shapes of pairs of points of any tetrahedron in $PG(3,2)$.

\begin{proposition}\label{prop:all-shapes}
The set $\boldsymbol{\sigma}=\{\ls, \hl, \ang, \alt, \ax, \dc,
\dm\}$ provides a complete list of shapes of pairs of points of
an arbitrary tetrahedron in $PG(3,2)$.
\end{proposition}

\begin{proof}
It is not difficult to see that each of the sets $\ls$, $\hl$,
$\ang$, $\alt$, $\ax$, $\dc$, and $\dm$ contains pairs of the
same shape. On the other hand, different members of
$\boldsymbol\sigma$ consist of pairs of different shape. Since
$$ |\ls|+|\hl|+|\ang|+|\alt|+|\ax|+|\dc|+|\dm|
=55=\binom{10}{2}+10=|\boldsymbol{\Pi}(T)|,$$ each pair of
points of $T$ belongs to precisely one element of
$\boldsymbol{\sigma}$. This implies that the set
$\boldsymbol{\sigma}$ is the complete set of shapes of pairs of
points of $T$.
\end{proof}

From among the shapes of $3$-element point sets only two will
be important for us -- lines (of course) and circles. We define
a \textit{circle} in a tetrahedron $T$ as any subset
$\{m_1,m_2,m_3\}$ of $T$ which consists of three distinct
midpoints sharing a triangle. It means that there exist three
corner points $c_1$, $c_2$, and $c_3$ of $T$ such that
$m_1=c_1+c_2$, $m_2=c_2+c_3$, and $m_3=c_3+c_1$. In particular,
$m_1+m_2+m_3=0$, so $\{m_1,m_2,m_3\}$ is a line of $PG(3,2)$.
Clearly, all circles have the same shape, but the shape of a
circle in $T$ is different from that of a line in $T$.

\begin{lemma}\label{lm:circle-line}
For arbitrary points $x$, $y$, and $z$ of a tetrahedron
$T$ in $PG(3,2)$, the equality $x+y+z=0$ holds if and only
if $\{x,y,z\}$ is a line or a circle of $T$.
\end{lemma}

\begin{proof}
If $\{x,y,z\}$ is a line of $T$ or a circle of $T$, then indeed
$x+y+z=0$. For the converse let $x$, $y$, and $z$ be arbitrary
points of $T$ such that $x+y+z=0$. Since all of them are
non-zero vectors, they must be pairwise distinct. Consider any
two of them, say $x$ and $y$. If they are collinear, then
$\{x,y,z\}$ is clearly a line.  If they are non-collinear, then
$\{x,y\}$ is either an angle, an axis, or an altitude. In the
latter two cases, $z=x+y$ is not a point of $T$. If $\{x,y\}$
is an angle, then $z$ is the third midpoint of the triangle
spanned by the corner points of the lines $\langle x\rangle$
and $\langle y\rangle$. In other words, $\{x,y,z\}$ is a circle
of $T$.
\end{proof}

\section{Transitions through (2,2)-poles}\label{sec:trans}

The next two sections provide a background material for the
study of conflicts of tetrahedral flows on edge-cuts of size
$4$ and $5$. For this purpose it is convenient to partition the
semiedges of $4$-poles and $5$-poles arising from the removal
of these cuts into two connectors of size $2$ and, in the
latter case, one additional residual semiedge. After choosing
an input connector we can follow how input values of a
tetrahedral flow transform into output values, and from this
information we can derive a transition relation for a dipole.

Consider an arbitrary (2,2)-pole $X=X(I,O)$ with input
connector $I=\{g_1,g_2\}$ and output connector $O=\{h_1,h_2\}$,
and let $T$ be a fixed tetrahedron in $PG(3,2)$. We say that
$X$ \textit{has a transition}
$$\{x,y\}\to\{x',y'\}$$
or that $\{x,y\}\to\{x',y'\}$ is a \textit{transition through}
$X$, if there exists a $T$-flow $\phi$ on $X$ such that
$\{\phi(g_1),\phi(g_2)\}=\{x,y\}$ and
$\{\phi(h_1),\phi(h_2)\}=\{x',y'\}$. If $X$ admits both
transitions $\{x,y\}\to\{x',y'\}$ and $\{x',y'\}\to\{x,y\}$, we
write
$$\{x,y\}\oto\{x',y'\}.$$
The set of all transitions through $X$ forms a binary relation
$\T_{\boldsymbol{\Pi}}(X)$ on the set
$\boldsymbol{\Pi}=\boldsymbol{\Pi}(T)$ of point pairs of $T$.

For convenience, we often refer to the symbol
$\{x,y\}\to\{x',y'\}$ as a \textit{transition} even without a
reference to a particular $(2,2)$-pole, as opposed to a
transition \textit{through} a $(2,2)$-pole defined above. There
is no danger of confusion.

From the Kirchhoff law we deduce that for each transition
$\{x,y\}\to\{x',y'\}$ through a $(2,2)$-pole we have
$x+y=x'+y'$. This value will be called the \textit{trace} of
the transition. A transition whose trace is $0$ is said to be
\textit{vanishing}. A non-vanishing transition
$\{x,y\}\to\{x',y'\}$ is \textit{collinear} if both pairs
$\{x,y\}$ and $\{x',y'\}$ are collinear, that is, if there
exist lines $\ell$ and $\ell'$ of $T$ such that
$\{x,y\}\subseteq\ell$ and $\{x',y'\}\subseteq\ell'$.

Each transition $\{x,y\}\to\{x',y'\}$ between pairs of points
induces the transition between their shapes. In this way we
obtain a \textit{transition} $\mathtt{s}\to\mathtt{t}$, where
$\mathtt{s}$ is the shape of $\{x,y\}$, $\mathtt{t}$ is the
shape of $\{x',y'\}$, and $\mathtt{s}, \mathtt{t}\in\{\ls, \hl,
\ang, \alt, \ax, \dc, \dm\} = \boldsymbol{\sigma}$. Since our
transitions are derived from flows, in most cases there is no
need to distinguish between the shapes $\dc$ and $\dm$ as both
of them represent the zero total flow through a connector. This
permits us to merge the shapes $\dc$ and $\dm$ into a single
shape $\dc\cup\dm$, which we denote by $\dpt$ and call the
\textit{double point}. Accordingly, we obtain the
\textit{merged set of shapes}
$$\boldsymbol{\Sigma}=\{\ls, \hl, \ang, \alt, \ax, \dpt \}$$
and denote the corresponding induced transition relation by
$\T_{\boldsymbol{\Sigma}}(X)$ or simply by $\T(X)$. In what
follows, this transition relation will become one of our main
tools for the investigation of cubic graphs that cannot be
covered with four perfect matchings.

\medskip

Before stating our next theorem we need one more definition. We
say that a transition $\{x,y\}\to\{x',y'\}$ is
\textit{stationary} if $\{x,y\}$ and $\{x',y'\}$ have the same
shape in $\Sh$. If all transitions through a $(2,2)$-pole $X$
are stationary, then $X$ itself is called \textit{stationary}.

\begin{theorem}\label{thm:admiss-trans}
All transitions through an arbitrary $(2,2)$-pole $X$ are
stationary except possibly those of the form $\ls\to\ang$ or
$\ang\to\ls$. In particular, all transitions involving a
half-line, an altitude, an axis, or a double point must have
one of the following forms:
$$ \hl\to\hl, \quad \alt\to\alt,
\quad \ax\to\ax,\quad \text{and} \quad \dpt\to\dpt. $$
\end{theorem}

\begin{proof}
Consider an arbitrary transition $\{x,y\}\to\{x',y'\}$ through
a $(2,2)$-pole $X$, and let $\phi$ be a $T$-flow on $X$ that
induces it. Kirchhoff's law implies that $x+y=x'+y'$, which
means that $x+y$ and $x'+y'$ must have the same weight. If
$|x+y|=2$, then each of $\{x,y\}$ and $\{x',y'\}$ is either an
angle or a line segment. Thus if the transition
$\{x,y\}\to\{x',y'\}$ with $|x+y|=2$ is not stationary, then it
has the form $\ang\to\ls$ or $\ls\to\ang$.

It remains to verify that all the remaining transitions through
$X$ are stationary. If $|x+y|=0$, then $x=y$  and $x'=y'$, and
we obtain the transition $\dpt\to\dpt$. If $|x+y|=1$, then
$x+y=x'+y'$ is a corner point, and therefore both $\{x,y\}$ and
$\{x',y'\}$ must be half-lines. Similarly, if $|x+y|=3$, then
$x+y$ is the centre of a triangle, which means that both
$\{x,y\}$ both $\{x',y'\}$ are altitudes. Finally, if
$|x+y|=4$, then $\{x,y\}$ both $\{x',y'\}$ are necessarily
axes. This yields the transitions $\hl\to\hl$, $\alt\to\alt$,
and $\ax\to\ax$, respectively.
\end{proof}

The previous theorem implies that the transition relation
$\T(X)$ of every $(2,2,)$-pole $X$ is contained in the set
\begin{align}\label{eq:A}
\mathcal{A}=\{ &\dpt\to\dpt, \hl\to\hl, \alt\to\alt, \ax\to\ax, \nonumber \\
               &\ang\to\ang, \ang\to\ls, \ls\to\ang, \ls\to\ls \}.
\end{align}
The elements of $\mathcal{A}$ will be called \textit{admissible
transitions}.

\begin{remark}\label{rem:admiss}
{\rm It is obvious that Theorem~\ref{thm:admiss-trans} can be
substantially strengthened as Kirchhoff's law provides
additional restrictions to admissible transitions. The
following statements hold for any transition
$\{x,y\}\to\{x',y'\}$ through an arbitrary $(2,2)$-pole $X$.
Their proofs are easy and therefore are left to the reader.
\begin{itemize}
\item[(i)] If $\{x,y\}\to\{x',y'\}$ has the form $\hl\to\hl$,
    then $\{x,y\}$ and $\{x',y'\}$ are half-lines with the
    same target, not necessarily identical.

\item[(ii)] If $\{x,y\}\to\{x',y'\}$ has the form $\ls\to\ls$,
    then $\{x,y\}=\{x',y'\}$.

\item[(iii)] If $\{x,y\}\to\{x',y'\}$ has the form
    $\ang\to\ls$, then $\{x',y'\}$ is the segment of the
    line of the opposite to the angle $\{x,y\}$ in the
    triangle containing $\{x',y'\}$. To be more precise, if
    $\{x,y\}=\{c_1+c_2, c_1+c_3\}$, then
    $\{x',y'\}=\{c_2,c_3\}$ for suitable corner points
    $c_1$, $c_2$, and $c_3$.

\item[(iv)] If $\{x,y\}\to\{x',y'\}$ has the form $\ang\to\ang$,
    then either $\{x,y\}=\{x',y'\}$, or $\{x,y\}$ and
    $\{x',y'\}$ are opposite angles of the rectangle formed
    by the union of triangles determined by $\{x,y\}$ and
    $\{x',y'\}$, respectively. The value $x+y=x'+y'$ is the
    midpoint of the line forming a diagonal of the
    rectangle.

\item[(v)] If $\{x,y\}\to\{x',y'\}$ has the form $\alt\to\alt$,
    then $\{x,y\}$ and $\{x',y'\}$ are altitudes of the
    same triangle.

\item[(vi)] There are no restrictions on transitions of the form
    $\ax\to\ax$ and $\dpt\to\dpt$. \qed
\end{itemize} }
\end{remark}

It is natural to ask whether all the admissible transitions can
actually occur in some dipole. The answer is positive but we
defer it until Example~\ref{ex:decol-trans} which is preceded
by the next two theorems.

\medskip

Consider a cubic graph $G$ with perfect matching index greater
than $4$, pick two adjacent vertices $u$ and $v$, and form the
$(2,2,)$-pole $D(I,O)=G_{uv}$. Our next theorem reveals the
fundamental property of the dipole constructed in the just
described way: for every non-vanishing transition
$\{x,y\}\to\{x',y'\}$ through $D(I,O)$ at most one of the pairs
$\{x,y\}$ is collinear.

\begin{theorem}\label{thm:decol-trans}
Let $G$ be a cubic graph with $\pi(G)\ge 5$, let $u$ and $v$ be
adjacent vertices of~$G$, and let $D(I,O)=G_{uv}$. If
$\{x,y\}\to\{x',y'\}$ is an arbitrary transition through $D$,
then at most one of the pairs $\{x,y\}$ and $\{x',y'\}$ is
collinear. In other words, $D$ has no collinear transition.
\end{theorem}

\begin{proof}
If $\{x,y\}\to\{x',y'\}$ is a vanishing transition, then both
pairs $\{x,y\}$ and $\{x',y'\}$ are degenerate and hence not
collinear. Assume that $\{x,y\}\to\{x',y'\}$ is a non-vanishing
transition through $D$, and let $\phi$ be a $T$-flow on $D$
which induces it. Suppose to the contrary that both pairs
$\{x,y\}$ and $\{x',y'\}$ are collinear. It follows that there
exist lines $\ell$ and $\ell'$ in $T$ such that
$\{x,y\}\subseteq\ell$ and $\{x',y'\}\subseteq\ell'$. Since
$x+y=x'+y'=t$ for some point $t$ of $T$, we see that
$t\in\ell\cap\ell'$. If we extend the flow $\phi$ on $D$ to the
entire graph $G$ by assigning the value $t$ to the edge $uv$,
the edges around $u$ will be properly coloured from the line
$\ell$ and those around $v$ will be properly coloured from
$\ell'$. It follows that $G$ has a $T$-flow, contradicting
Theorem~\ref{thm:1-1}. Hence $X$ has no collinear transition.
\end{proof}

Theorem~\ref{thm:decol-trans} implies that every $(2,2)$-pole
$G_{uv}$ obtained from a cubic graph $G$ with $\pi(G)\ge 5$ by
removing two adjacent vertices and arranging the dangling edges
in the usual manner behaves like a \textit{collinearity
destroying gadget}: if the input pair of a transition is
collinear, then the  output pair must be non-collinear, and
vice versa.  More generally, every  $(2,2)$-pole $X$ which has
no collinear transitions will be called a  \textit{collinearity
destroying dipole}, or briefly a \textit{decollineator}.

\medskip

The following result provides a characterisation of
decollineators.

\begin{theorem}\label{thm:char-decol}
The following statements are equivalent for an arbitrary
$(2,2)$-pole $X$.
\begin{itemize}
\item[{\rm (i)}] $X$ is a decollineator, that is, $X$
    admits no collinear transition.

\item[{\rm (ii)}] $X$ has no transitions of the form
    $\ls\to\ls$ or $\hl\to\hl$.

\item[{\rm (iii)}] The cubic graph $G$ created from $X$ by
    adding to $X$ two adjacent vertices and attaching each
    of them to a connector of $X$ has $\pi(G)\ge5$.
\end{itemize}
\end{theorem}

\begin{proof}
(i) $\Rightarrow$ (ii): Let $X$ be a decollineator.
Theorem~\ref{thm:admiss-trans} implies that the only possible
transitions through $X$ are the admissible transitions
constituting the set $\mathcal{A}$. Of them, only $\ls\to\ls$
and $\hl\to\hl$ are collinear. So $X$ has no transitions of the
form $\ls\to\ls$ or $\hl\to\hl$.

\medskip

(ii) $\Rightarrow$ (iii): Let $X$ be a $(2,2)$-pole that admits
no transitions of the form $\ls\to\ls$ or $\hl\to\hl$, and let
$G$ be the cubic graph formed from $X$ by adding two adjacent
vertices and attaching each of them to a connector of $X$. By
Theorem~\ref{thm:1-1} it is sufficient to show that $G$ has no
$T$-flow. If $G$ had one, say $\phi$, then the three edges
around $u$ and $v$ would receive values from lines $\ell_u$ and
$\ell_v$ of $T$, respectively. Under the induced flow on $X$,
the semiedges of the input connector receive values $x$ and $y$
from $\ell_u$,  and those of the output connector receive
values $x'$ and $y'$ from $\ell_v$. It means that $X$ has a
transition $\{x,y\}\to\{x',y'\}$ which is collinear, contrary
to the assumption.

\medskip

(iii) $\Rightarrow$ (i): This implication follows directly from
Theorem~\ref{thm:decol-trans}.
\end{proof}

Theorems~\ref{thm:admiss-trans}-\ref{thm:char-decol} combined
readily imply the following.

\begin{corollary}\label{cor:D}
Every decollineator $D$ has its transition relation
$\T(D)$ contained in the set
 \begin{align}\label{eq:D}
 \mathcal{D}=\{\dpt\to\dpt, \alt\to\alt, \ax\to\ax,
               \ang\to\ang, \ang\to\ls, \ls\to\ang\}.
 \end{align}
\end{corollary}

The next example shows that each admissible transition occurs
in some $(2,2)$-pole.

\begin{example}\label{ex:decol-trans}{\rm
The decollineator $G_{uv}$ where $G$ is the Petersen graph has all
admissible non-collinear transitions. The transition relation
of the $(2,2)$-pole consisting of two adjacent vertices, with
input semiedges attached to one vertex and the output semiedges
attached to the other vertex, is the set $\{\ls\to\ls,
\hl\to\hl\}$ comprising the two collinear transitions.
}\end{example}

We finish this section with two results that reveal interesting
behaviour of bipartite cubic graphs.

\begin{proposition}\label{prop:bip-3pole}
Let $G$ be a bipartite cubic graph and $v$ an arbitrary vertex
of $G$. Then for every tetrahedral flow on the $3$-pole $G_v$
the flow values of the three semiedges of $G_v$ form a line
of~the tetrahedron.
\end{proposition}

\begin{proof}
Let $x$, $y$, and $z$ be the values assigned to the semiedges
of $G_v$. From Kirchhoff's law we know that $x+y+z=0$. By
Lemma~\ref{lm:circle-line}, the triple $\{x,y,z\}$ is either a
line or a circle of $T$. If $\{x,y,z\}$ is a circle, then the
edges carrying a value of weight $1$ must form a $2$-factor $F$ of
$G-v$. Since $G-v$ has an odd number of vertices, $F$ contains
an odd circuit, and hence $G$ is not bipartite, contrary to the
assumption. Therefore $\{x,y,z\}$ is a line, as claimed.
\end{proof}

\begin{proposition}\label{prop:bip22pole}
Let $G$ be an arbitrary bipartite cubic graph, and let $u$ and
$v$ be adjacent vertices of $G$. Then $G_{uv}$ is a stationary
$(2,2)$-pole.
\end{proposition}

\begin{proof}
Suppose to the contrary that $G_{uv}$ is not stationary.
Theorem~\ref{thm:admiss-trans} then implies that the dipole
$G_{uv}$ admits a transition $\ang\to\ls$ or its reverse.
Without loss of generality we may assume that $G_{uv}$ admits
$\ang\to\ls$. Let $\phi$ be a $T$-flow of $G_{uv}$ that induces
this transition. Then $\phi$ gives rise to a $T$-flow of $G-u$
under which the values on the semiedges of $G_u$ form a circle.
This contradicts Proposition~\ref{prop:bip-3pole}.
\end{proof}

\section{Weighted transitions through (2,2;1)-poles}
\label{sec:wtrans}

Consider an arbitrary $(2,2;1)$-pole $Y=Y(I,O;r)$ with input
connector $I=\{g_1,g_2\}$, output connector $O=\{h_1,h_2\}$,
and residual semiedge $r$. Let $T$ be a tetrahedron in
$PG(3,2)$, let $\{x,y\}$ and $\{x',y'\}$ be elements of
$\boldsymbol{\Pi}(T)$, and finally let $i\in\{1,2\}$. We say
that $X$ has a \textit{transition}
$$\{x,y\}\overset i\to\{x',y'\}$$
if there exists a $T$-flow $\phi$ on $Y$ such that
$\{\phi(g_1),\phi(g_2)\}=\{x,y\}$,
$\{\phi(h_1),\phi(h_2)\}=\{x',y'\}$, and the \textit{residual
value} $\phi(r)$ has weight $i$. For emphasis, such a
transition will be called a \textit{weighted transition}.
However, whenever the context permits it, the adjective
`weighted' will be dropped.

Every weighted transition $\{x,y\}\overset i\to\{x',y'\}$
between pairs of points induces a transition
$\mathtt{s}\overset i\to\mathtt{t}$ where $\mathtt{s}$ is the
shape of $\{x,y\}$ and $\mathtt{t}$ is the shape of
$\{x',y'\}$. Although the weight~$i$ of a transition
$\{x,y\}\overset i\to\{x',y'\}$ is uniquely determined by
$\{x,y\}$ and $\{x',y'\}$, the same might not hold for
$\mathtt{s}\overset i\to\mathtt{t}$. The reason resides in the
fact that the same shapes may correspond to different geometric
positions of the corresponding pairs of points within the
tetrahedron. Hence, the weight of a transition
$\mathtt{s}\overset i\to\mathtt{t}$ carries an information that
cannot be immediately recovered from $\mathtt{s}$ and
$\mathtt{t}$.

Most of the terminology developed in Section~\ref{sec:trans}
for unweighted transitions directly modifies to weighted
transitions. In particular, a \textit{vanishing} weighted
transition is one where either the input pair or the output
pair is a double point.

Our first aim is to study the weighted transitions through a
$(2,2;1)$-pole $G_{uwv}$ obtained from a bipartite cubic graph
$G$ by removing a path $uwv$, where the input $I=\{g_1,g_2\}$
and the output $O=\{h_1,h_2\}$ are formed by the semiedges
formerly incident with $u$ and $v$, respectively, and the
residual semiedge is the one formerly incident with $w$. For
brevity, such a $(2,2;1)$-pole will be called
\textit{bipartite}. Analogously we say that a $(2,2)$-pole
$G_{uv}$ arising from a bipartite cubic graph by removing two
adjacent vertices is \textit{bipartite}.

\medskip

We need two lemmas.

\begin{lemma}\label{lm:bip-balance}
Let $G$ be a bipartite graph and let $uwv$ be a path of
length~$2$ in~$G$. For an arbitrary tetrahedral flow on the
$(2,2;1)$-pole $G_{uwv}$, the total number of non-residual
semiedges receiving a value of weight $2$ does not exceed $2$.
\end{lemma}

\begin{proof}
Let $\{A,B\}$ be a bipartition of $G$. Since $G$ is regular, we
have $|A|=|B|=m$ for some integer $m$. Without loss of
generality we may assume that $u$ and $v$ belong to $A$ while
$w$ belongs $B$. Under any given $T$-flow on $G-\{u,w,v\}$,
each vertex of $G-\{u,w,v\}$ is incident with exactly one edge
that carries a value of weight $2$. Let $\bar p$ denote the
number of semiedges of $G_{uwv}$ formerly incident with the
vertex $p\in\{u,w,v\}$ that receive a flow value of weight $2$.
Counting the edges with a value of weight $2$ leaving $A$ and
those leaving $B$ yields
$$\bar u + \bar v + (m-2) = \bar w + (m-1),$$
which simplifies to
$$\bar u + \bar v = \bar w + 1.$$
However, $\bar w\le 1$, so $\bar u + \bar v \le 2$, which is
equivalent to the statement of this lemma.
\end{proof}

\begin{lemma}\label{lm:semi-parity}
In an arbitrary $k$-pole $Q$ endowed with a tetrahedral flow,
the number of semiedges carrying a value of weight $2$ has the
same parity as $k$.
\end{lemma}

\begin{proof}
This is a direct consequence of the fact that in a multipole
furnished with a tetra\-hedral flow every vertex is incident
with exactly one edge that carries a value of weight~$2$.
\end{proof}

The combination of Lemmas~\ref{lm:bip-balance}
and~\ref{lm:semi-parity} significantly restricts potential
transitions through a $(2,2;1)$-pole $G_{uwv}$ arising from a
bipartite graph $G$. In the next theorem we specify a set $\mathcal{B}$ of 
weighted transitions that can occur for such a $(2,2;1)$-pole. The subsequent 
Remark~\ref{rem:Hw} confirms that $\mathcal{B}$ is a minimal transition set 
with this property because it actually occurs as the transition relation for a 
suitable bipartite $(2,2;1)$-pole~$G_{uwv}$.

The diagram of $\mathcal{B}$ is depicted in Figure~\ref{fig:transB}. In this 
diagram, and in all other diagrams representing weighted transition relations, 
undirected connections between two shapes $\mathtt{s}$ and $\mathtt{t}$ indicate 
the existence of both the transition $\mathtt{s}\overset i\to\mathtt{t}$ and
$\mathtt{t}\overset i\to\mathtt{s}$. Transitions of weight $2$
are represented by bold lines.

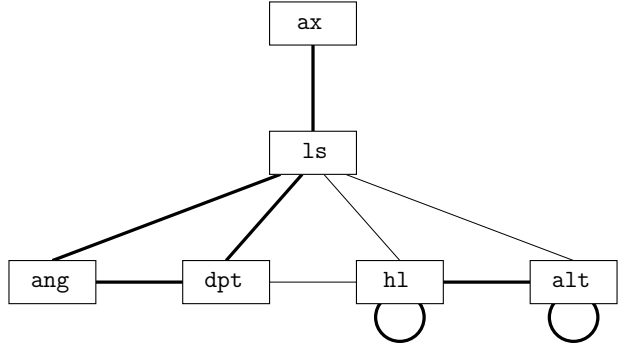
\begin{figure}[htbp]
\begin{center}
     \scalebox{0.45}
     {
       \input{prechodyB}
     }
\end{center}
  \caption{The set of transitions $\mathcal{B}$ through bipartite $(2,2;1)$-poles}
  \label{fig:transB}
\end{figure}

\begin{theorem}\label{thm:wtrans-bip}
Let $G$ be a bipartite cubic graph and let $uwv$ be a path
of length~$2$ in~$G$. Then $\T(G_{uvw})\subseteq \mathcal{B}$.
\end{theorem}

% \begin{theorem}\label{thm:wtrans-bip}
% Let $G$ be a connected bipartite graph and let $uwv$ be a path
% of length~$2$ in~$G$. Then $\T(G_{uvw})\subseteq \mathcal{B}$.
% \end{theorem}

\begin{proof}
Let $\{x,y\}\overset i\to\{x',y'\}$ be an arbitrary transition
through $G_{uwv}$ and let $z$ be the corresponding residual
value. Since $\mathcal{B}$ is a symmetric relation, from each pair of 
transitions
$\{x,y\}\overset i\to\{x',y'\}$ and $\{x',y'\}\overset i\to\{x,y\}$  through 
$G_{uvw}$, whenever they exist, it is sufficient to check
the one with $|x|+|y|\ge |x'|+|y'|$.

% Note that $\T(G_{uvw})$ is symmetric relation. Therefore we will without 
% loss of generality assume that $|x|+|y|\ge |x'|+|y'|$ and with each ordered pair 
% we will also include its inverse to $\T(G_{uvw})$!!!.

To prove the required inclusion we consider three main
cases depending on the weight of $x$ and $y$. Throughout the
proof we therefore distinguish between the shapes $\dc$ and
$\dm$. In our discussion we will say that a weighted transition
$\mathtt{s}\overset i\to\mathtt{t}$ through $G_{uwv}$ is
\textit{valid} if it occurs in~$\mathcal{B}$.

\medskip\noindent
\textbf{Case 1.} $|x|=|y|=2$.
From Lemma~\ref{lm:bip-balance} we get that $|x'|=|y'|=1$.

\medskip\noindent
\textbf{Case 1.1.} $x=y$. In this case $x'+y'=z\ne 0$, so
$\{x',y'\}$ is a line segment. Hence, we have a valid
transition $\dm\overset 2\to\ls$.

\medskip\noindent
\textbf{Case 1.2.} $x\ne y$. This time $\{x,y\}$ is either an
axis or an angle. First assume that $\{x,y\}$ is an axis. Then
$x'+y'\ne 0$ for otherwise we would have $z=x+y$, where $x+y$
is not a point of the tetrahedron. Therefore $\{x',y'\}$ is a
line segment. From Lemma~\ref{lm:semi-parity} we infer that
$|z|=2$, which yields a valid transition $\ax\overset 2\to\ls$.
Next assume that $\{x,y\}$ is an angle. Recall that $\{x',y'\}$
is a pair of corner points. If $x'=y'$, then we obtain the
valid transition $\ang\overset 2\to\dc$ with residue $z=x+y$,
which is the third point of the circle $\{x,y,x+y\}$. If $x'\ne
y'$, then $\{x',y'\}$ is a line segment, and by
Lemma~\ref{lm:semi-parity} the residual value is a midpoint. We
have thus obtained the transition $\ang\overset 2\to\ls$, which
again is valid.

\medskip\noindent
\textbf{Case 2.} $|x|=|y|=1$.

\medskip\noindent
\textbf{Case 2.1.} $x=y$. If $z$ is a midpoint, then
Lemma~\ref{lm:semi-parity} implies that $|x'|=|y'|$. 
If $|x'|=|y'|=1$, then $x'+y'=z\ne 0$, so $\{x',y'\}$
is a line segment. Thus we have obtained the
transition $\dc\overset 2\to\ls$, which is again valid.

If $z$ is a corner point, then from Lemma~\ref{lm:semi-parity}
we infer that $|x'|\ne |y'|$, which in turn implies that   $|x|+|y|< |x'|+|y'|$, 
a contradiction. 

\medskip\noindent
\textbf{Case 2.2.} $x\ne y$. In this case $\{x,y\}$ is a line
segment. If $z$ is a corner point, we again get a cotradiction with $|x|+|y|< 
|x'|+|y'|$. Assume therefore that $z$ is a midpoint. Then $|x'|=|y'|=1$. We 
claim that in
this situation $x'=y'$. If not, then $\{x',y'\}$ is a line
segment, which gives rise to the transition $\ls\overset
2\to\ls$. However, any $T$-flow on $G_{uwv}$ corresponding to
this transition would induce a $T$-flow on $G_w$ under which
the values on the semiedges of $G_w$ form a circle
$\{x+y,z,x'+y'\}$. This contradicts
Proposition~\ref{prop:bip-3pole}. Hence $x'=y'$, which again
gives rise to a valid transition $\ls\overset 2\to\dc$.

\medskip\noindent
\textbf{Case 3.} $|x|\ne |y|$. This means that $\{x,y\}$ is
either an altitude or a half-line. If $z$ is a midpoint,
Lemma~\ref{lm:semi-parity} yields that $|x'|\ne |y'|$, which
implies that $\{x',y'\}$ is either a half-line or an altitude
as well. The resulting transitions are $\hl\overset 2\to\hl$,
$\alt\overset 2\to\alt$, and $\hl\overset 2\to\alt$ or its
reverse, all of them valid.

If $z$ is a corner point, then $|x'|=|y'|$, by
Lemma~\ref{lm:semi-parity}. However, Lemma~\ref{lm:bip-balance}
excludes the possibility that $|x'|=|y'|=2$, so $|x'|=|y'|=1$.
If $x'=y'$, then $\{x,y\}$ cannot be an altitude, because
otherwise $z=x+y$ whence $z$ would not be a point of $T$.
Therefore $\{x,y\}$ is a half-line and we obtain the valid
transition $\hl\overset 1\to\dc$. Finally, if $x'\ne y'$, then
$\{x',y'\}$ is a line segment, which yields transitions
$\hl\overset 1\to\ls$ and $\alt\overset 1\to\ls$, both of them
valid.

To summarise, we have covered all the possibilities for the
input pair $\{x,y\}$ and verified that all the transitions from
$\{x,y\}$ through $G_{uvw}$, where $G$ is a bipartite graph,
are contained in $\mathcal{B}$. Moreover, every transition
contained in $\mathcal{B}$ occurs as a valid transition at
least once. This completes the proof.
\end{proof}

\begin{remark}\label{rem:Hw}
{\rm The inclusion $\T(G_{uvw})\subseteq\mathcal{B}$ can
actually be achieved with equality for suitable bipartite
graphs. This situation takes place, for example, for both the
complete bipartite graph $K_{3,3}$ and the Heawood graph, the
unique cubic bipartite graph of girth 6 on 14 vertices
\cite{FC}. In the former case, verification can be easily done
directly, but for the Heawood graph a computer was necessary.

It may be interesting to mention that if the shapes $\dc$ and
$\dm$ are distinguished, then the transition \mbox{$\dc\overset
2\to\ls$} does not exist for $K_{3,3}$, while
\mbox{$\dm\overset 2\to\ls$} does. On the other hand, the
Heawood graph admits both transitions \mbox{$\dc\overset
2\to\ls$} and \mbox{$\dm\overset 2\to\ls$}. In some cases,
distinguishing between $\dc$ and $\dm$ does make sense and can be used 
for the construction of graphs with perfect matching index at least $5$.} \qed
\end{remark}

\section{Composing dipoles and their transitions}
\label{sec:composing}

In order to be able to construct rich families of graphs with
perfect marching index at least $5$ we will employ several
operations on dipoles as well as on their transition relations.
The definitions come with no surprise, but we include them in
order to avoid ambiguity.

Given an $(2,2;d_1)$-pole $M_1$ and an  $(2,2;d_2)$-pole $M_2$,
we can construct an $(2,2;d_1+d_2)$-pole $M_1\circ M_2$, the
\textit{join} of  $M_1$ and $M_2$, by joining the output
connector of $M_1$ with the input connector of $M_2$. The input
and the output of $M_1\circ M_2$ are inherited from $M_1$ and
$M_2$, respectively. The join is clearly associative, which
means that $(M_1\circ M_2)\circ M_3=M_1\circ (M_2\circ M_3)$.
If both $M_1$ and $M_2$ are $(2,2)$-poles, we will also refer
to $M_1\circ M_2$ as their \textit{composition}.

Given two $(2,2;1)$-poles $M_1$ and $M_2$, we can construct a
new $(2,2;1)$-pole $M_1\odot M_2$, the \textit{composition} of
$M_1$ and $M_2$, by attaching the two residual semiedges of
$M_1\circ M_2$ to a new vertex, say $v$, and by adding a new
dangling edge incident with $v$ which will become the residual
edge of $M_1\odot M_2$.

Technically speaking, the dipoles $M_1\circ M_2$ and $M_1\odot
M_2$ are not uniquely determined by $M_1$ and $M_2$ as they
depend on the ordering of semiedges involved in the operation.
However, all our subsequent
statements concerning $M_1\circ M_2$ or $M_1\odot M_2$ will
equally hold for both possible outcomes of either of these
operations.

Similarly to dipoles, we can also compose their transition
relations. As expected, unweighted transitions
$\mathtt{p}\to\mathtt{s}$ and $\mathtt{s}\to\mathtt{t}$ of
dipoles $M_1$ and $M_2$, respectively, give rise to the
transition $\mathtt{p}\to\mathtt{t}$ of $M_1\circ M_2$.
Conversely, a transition $\mathtt{p}\to\mathtt{q}$ through
$M_1\circ M_2$ occurs only when there exist transitions
$\mathtt{p}\to\mathtt{s}$ through $M_1$ and
$\mathtt{s}\to\mathtt{t}$ through $M_2$ for a suitable shape
$\mathtt{s}\in\Sh$. If exactly one of the transitions
$\mathtt{p}\to\mathtt{s}$ and $\mathtt{s}\to\mathtt{t}$ is
weighted, say $\mathtt{s}\overset i\to\mathtt{t}$, then its
weight is inherited into the resulting transition
$\mathtt{p}\overset i\to\mathtt{t}$. If both tran\-sitions are
weighted, say $\mathtt{p}\overset i\to\mathtt{s}$ and
$\mathtt{s}\overset j\to\mathtt{t}$, then the composite
transition $\mathtt{p}\overset k\to\mathtt{s}$ through
$M_1\odot M_2$ is defined if and only $i+j\le 3$, in which case
its weight is $k=3-ij$. The last rule is a consequence of the
fact that each vertex of a $(2,2;1)$-pole carrying a $T$-flow
is incident with two edges with value of weight $1$ and one
edge with value of weight $2$. The same rules apply not just to
transition relations of dipoles but also to arbitrary sets of
transitions. So, if $\T(M_1)\subseteq\mathcal{R}_1$ and
$\T(M_2)\subseteq\mathcal{R}_2$ we can conclude that
$\T(M_1\circ M_2)\subseteq\mathcal{R}_1\circ\mathcal{R}_2$ or
$\T(M_1\odot M_2)\subseteq\mathcal{R}_1\odot\mathcal{R}_2$,
depending on which composition operation applies to $M_1$
and~$M_2$.

\medskip

The next two lemmas are easy but useful.

\begin{lemma}\label{lm:trans-composition}
The following statements hold:
\begin{itemize}
 \item[{\rm (i)}] $\T(M_1\circ M_2)=\T(M_1)\circ
 \T(M_2)$ for any two $(2,2)$-poles $M_1$ and $M_2$.
 \item[{\rm (ii)}] $\T(M_1\odot M_2)=\T(M_1)\odot
 \T(M_2)$ for any two $(2,2;1)$-poles $M_1$ and $M_2$.
\end{itemize}
\end{lemma}

\begin{lemma}\label{lm:trans-comp_w_stationary}
If $X$ is an arbitrary $(2,2)$-pole or $(2,2;1)$-pole and $Y$
is a stationary $(2,2)$-pole, then $\T(X\circ Y)\subseteq
\T(X)$ and $\T(Y\circ X)\subseteq \T(X)$.
\end{lemma}

Composition of dipoles can be conveniently utilised in
constructing dipoles with various useful transition properties.
Preventing collinear transitions, which is defining property of
decollineators, is one of them. Another important type of a
dipole is $(2,2)$-pole which does not permit any transition of
the form $\ang\to\ang$. We call it a \textit{deangulator}. The
smallest example of a deangulator is the $(2,2)$-pole of the
form $G_{uv}$ obtained from $G=K_4$, the complete graph on four
vertices, where $uv$ is any edge of $K_4$. Another example can
be constructed from the Petersen graph by severing two edges at
distance~$2$ and forming each connector from the half-edges of
the same edge. Restricting the consideration to the weights of
points it is easy to verify that neither of these dipoles has a
transition $\{x,y\}\to\{x',y'\}$ where $|x|=|y|=|x'|=|y'|=2$.
Consequently, no transition of the form $\ang\to\ang$ through
them is possible.

\medskip

The next proposition can be applied to the construction of new
decollineators and thereby, having in mind
Theorem~\ref{thm:char-decol}~(iii), to the construction of new
cubic graphs with $\pi(G)\ge 5$.

\begin{proposition}\label{prop:2decol}
If $D_1$ and $D_2$ are decollineators and $U_1$ and $U_2$ are
deangulators, then $D_1\circ U_i\circ D_2$ is a decollineator
and $U_1\circ D_i\circ U_2$ is a deangulator for each
$i\in\{1,2\}$.
\end{proposition}

\begin{proof}
By Theorem~\ref{thm:char-decol} it suffices to prove that the
$(2,2)$-pole $D=D_1\circ U_i\circ D_2$ has no transitions of
the form $\hl\to\hl$ and $\ls\to\ls$. The former transition is
immediately excluded by Theorem~\ref{thm:admiss-trans}. In
order to exclude the latter, consider an arbitrary transition
through $D$ starting with a line segment. Applying the
assumptions and Theorem~\ref{thm:admiss-trans} to $D_1$, $U$,
and $D_2$ we conclude that the only feasible sequence of
transitions through them is $\ls\to\ang\to\ls\to\ang$, where
the first transition is through $D_1$, the second is through
$U_i$, and the third is through $D_2$. The resulting transition
through $D$ is $\ls\to\ang$, which proves that $D$ has no
transition $\ls\to\ls$ as well. The proof for $U_1\circ
D_i\circ U_2$ is similar.
\end{proof}

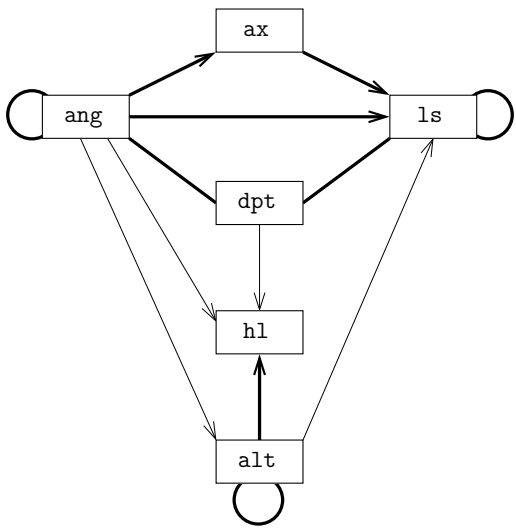
\begin{figure}[htbp]
\begin{center}
     \scalebox{0.45}
     {
       \input{prechodyDB}
     }
\end{center}
\caption{The set of weighted transitions $\mathcal{D}\circ\mathcal{B}$}
\label{fig:transDB}
\end{figure}

\begin{example}\label{ex:windmill}{\rm
We finish this section by exploring the transition relation of
a Halin fragment. Recall that a Halin fragment is a
$(2,2;1)$-pole of the form $F=D\circ B$ where $D$ is a
decollineator, which can be taken to be a $(2,2)$-pole $G_{uv}$
obtained from a cubic graph $G$ with $\pi(G)\ge 5$ by removing
a path $uv$ of length $1$, and $B$ is a $(2,2;1)$-pole
$K_{uwv}$ arising from a cubic bipartite graph $K$ by removing
a path $uwv$ of length~$2$. Clearly, $\T(D\circ
B)\subseteq\mathcal{D}\circ \mathcal{B}$. If $D$ is a
decollineator constructed from the Petersen graph and the
$(2,2;1)$-pole $B$ is created from the complete bipartite graph
$K_{3,3}$, then the resulting Halin fragment coincides with the
Petersen fragment  $F_{Ps}$, the basic building block of
treelike snarks. It can be shown that
$\T(F_{Ps})=\mathcal{D}\circ \mathcal{B}$.

With this information we can easily generalise the family of
windmill snarks intro\-duced in \cite{EM}. Take three  Halin
fragments $F_1$, $F_2$, and $F_3$ and form a cubic graph
$W=W(F_1,F_2,F_3)$ by first creating their composition
$F_1\circ F_2\circ F_3$, then attaching the three residual
semiedges to a new vertex, and finally by taking a closure of
the resulting $(2,2)$-pole. We claim that $\pi(W)\ge 5$. If $W$
had a covering with four perfect matchings, then the
corresponding $T$-flow would induce a directed closed walk of
length $3$ in the diagram of the transition relation
$\mathcal{D}\circ\mathcal{B}$ with the property that exactly
one of the transitions has weight $2$. A straightforward check
with the help of Figure~\ref{fig:transDB} reveals that such
a closed walk does not exist. Therefore $\pi(W)\ge 5$, as
claimed.

If the bipartite constituent $B_i$ of each of the three Halin
fragments $F_i=D_i\circ B_i$  is constructed from $K_{3,3}$,
then $W$  coincides with a windmill snark of Esperet and
Mazzuoccolo \cite{EM}. Moreover, if we choose each $F_i$ to be
the Petersen fragment, then $W$ becomes isomorphic to the snark
shown in Figure~\ref{fig:g34+46}~(left). It is therefore
appropriate to call the family of graphs of the form
$W(F_1,F_2,F_3)$, where $F_1$, $F_2$, and $F_3$ are arbitrary
Halin fragments, the \textit{generalised windmill snarks}. }
\end{example}

\section{Halin snarks}\label{sec:Halin}

In this section we put to use the knowledge gathered in the
previous sections and prove that Halin snarks cannot be covered
with four perfect matchings. We also prove that they are
nontrivial snarks whenever the building blocks employed for
their construction originate from cyclically $4$-edge-connected
graphs of girth at least $5$.

Recall that a Halin snark arises from a Halin graph $H=C\cup
S$, where $C$ is the perimeter circuit and $S$ is the inscribed
tree, by substituting the vertices of $C$ with Halin fragments.
As mentioned earlier, a Halin fragment is a $(2,2;1)$-pole
$F=D\circ B$ where $D$ is a decollineator and $B$ is a
bipartite $(2,2;1)$-pole. Let $v_0,v_1,\dots, v_{k-1}$ be the
vertices of $C$ arranged in a cyclic order, and let
$F_i=D_i\circ B_i$, with $i\in\{0, 1, 2,\ldots, k-1\}$, be a
Halin fragment that substitutes the vertex $v_i$ in the
construction. We now form a \textit{Halin snark}~$H^{\sharp}$
by, first, performing a junction of the output connector of
$F_i$ with the input connector of $F_{i+1}$ for each
$i\in\{0,1,\ldots, k-1\}$ taken modulo $k$; then, by adding in
the tree $S'=S-\{v_0,\ldots,v_{k-1}\}$; and, finally, by
attaching the residual semiedge of each $F_i$ to the unique
vertex $v_i'$ of $S'$ adjacent in $S$ to $v_i$. For
convenience, we identify the vertex of $F_i$ adjacent to $v_i'$
with $v_i$. As a result, $S$ is inherited from $H$ to $H^{\sharp}$
in its entirety, and so we can write $H^{\sharp}=C^{\sharp}\cup
S$ where $C^{\sharp}$ is a subgraph obtained from $H^{\sharp}$
by the removal of all $3$-valent vertices of $S$.

It is not difficult to show that every Halin snark $H^{\sharp}$
is indeed a snark, meaning that it admits no proper
$3$-edge-colouring. One can argue, for example, by using the
fact that in each Halin fragment $D_i\circ B_i$ the $(2,2)$-pole $D_i$ is
isochromatic \cite[p.~13]{ChS}. Therefore, for every $3$-edge-colouring 
of $H^{\sharp}$ -- which can be taken a nowhere-zero 
$\mathbb{Z}_2\times\mathbb{Z}_2$-flow on $H^{\sharp}$ -- the Kirchhoff law 
fails at each vertex $v_i$, forcing the zero value on the edge $v_iv_i'$ 
that connects $B_i$ to $S$. We omit the
details because in Theorem~\ref{thm:Halin(2,2;1)} we prove a
stronger result, namely that $\pi(H^{\sharp})\ge 5$.

According to our definition, Halin snarks need not be
nontrivial in general. The next proposition shows that their
nontriviality can easily be achieved if the building blocks are
by properly choosen.

\begin{proposition}\label{prop:nontriv-Halin}
If $H^{\sharp}$ is a Halin snark such that each Halin fragment
$F_i=D_i\circ B_i$ has both the decollineator $D_i$ and the
bipartite $(2,2;1)$-pole $B_i$ constructed from a cyclically
$4$-edge-connected cubic graph of girth at least $5$, then
$H^{\sharp}$ is cyclically $4$-edge-connected, with girth at
least $5$.
\end{proposition}

\begin{proof}
Let $H^{\sharp}$ be a Halin snark satisfying the assumptions.
It is easy to see that $H^{\sharp}$ has girth at least $5$. We
claim that $H^{\sharp}$ is cyclically $4$-edge-connected.
Obviously, none of the multipoles used in the construction has
a $k$-edge-cut with $k<4$ separating a subgraph containing a
cycle from the rest of $H^{\sharp}$. Moreover, all Halin graphs
are $3$-edge-connected \cite{Halin}. Thus, every
cycle-separating edge-cut is either entirely contained in one
of the multipoles -- and then it is an $l$-cut for $l\ge 4$ --
or its mapping to the original Halin graph cuts the perimeter
circuit twice, and thus is a cycle separating $m$-edge cut with
$m\ge 5$.
\end{proof}

\begin{figure}[htbp]
\begin{subfigure}{.5\textwidth}
  \centering
  \scalebox{0.45}{\input{prechodyM}}
\end{subfigure}
\begin{subfigure}{.5\textwidth}
  \centering
  \scalebox{0.45}{\input{prechodyMM}}
\end{subfigure}
\caption{
The transition relations $\mathcal{M}$ (left) and
$\mathcal{M}\odot\mathcal{M}$ and $\mathcal{M}'$ (right);
transitions added to $\mathcal{D}\circ\mathcal{B}$ are dotted,
those of $\mathcal{M}\odot\mathcal{M}$ that do not occur in
$\mathcal{M}'$ are dashed.} \label{fig:transM+MM}
\end{figure}
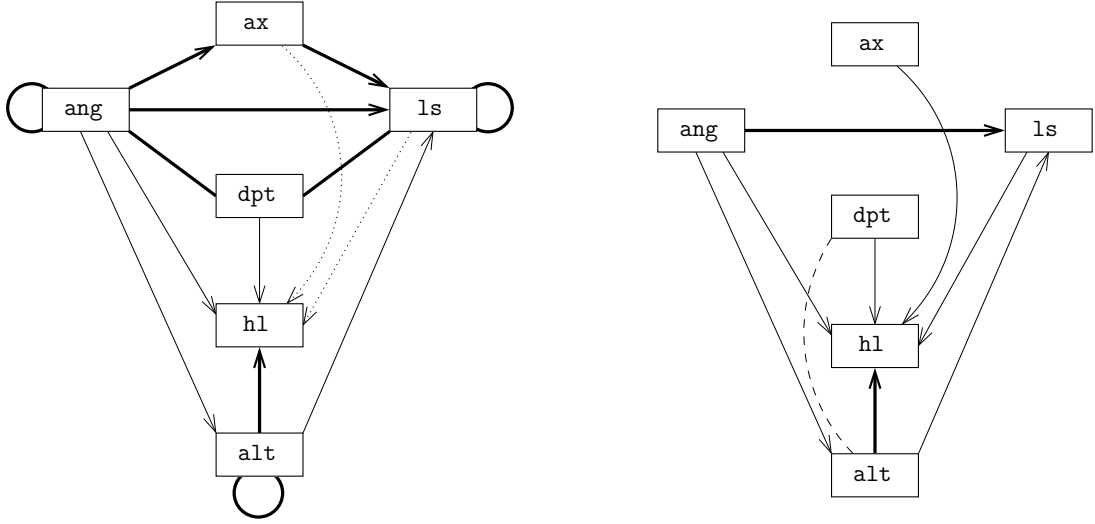

We now proceed to the proof that every Halin snark $H^{\sharp}$
has $\pi(H^{\sharp})\ge 5$. The idea is to modify the argument
of Example~\ref{ex:windmill} by engaging induction. For this
purpose we extend the weighted transition relation
$\mathcal{D}\circ\mathcal{B}$ to the set
\begin{align}
\mathcal{M}=\mathcal{D}\circ\mathcal{B}\cup
            \{\ax\overset 1\to\hl, \ls\overset 1\to\hl\}
\end{align}
whose diagram of $\mathcal{M}$ is displayed in
Figure~\ref{fig:transM+MM}~(left).

\medskip

The following property of $\mathcal{M}$ will be crucial for the
inductive proof.

\begin{proposition} \label{prop:M'}
If $M_1$ and $M_2$ are arbitrary $(2,2;1)$-poles such that
$\T(M_1)\subseteq\mathcal{M}$ and
$\T(M_2)\subseteq\mathcal{M}$, then $\T(M_1\odot
M_2)\subseteq \mathcal{M}\odot\mathcal{M}-\{\dpt\overset
1\oto\alt\}$.
\end{proposition}

\begin{proof}
If $\T(M_1)\subseteq\mathcal{M}$ and
$\T(M_2)\subseteq\mathcal{M}$, then clearly $\T(M_1\odot
M_2)\subseteq\mathcal{M}\odot\mathcal{M}$; the diagram of
$\mathcal{M}\odot\mathcal{M}$ is displayed in
Figure~\ref{fig:transM+MM}~(right). We show that the
transitions $\dpt\overset 1\to\alt$ and $\alt\overset
1\to\dpt$, which occur in $\mathcal{M}\odot\mathcal{M}$ and in
the diagram are represented by a dashed line, do not belong to
$\T(M_1\odot M_2)$.

Suppose to the contrary that $M_1\odot M_2$ admits the
transition $\dpt\overset 1\to\alt$, and let $\phi$ be the
corresponding $T$-flow. Since the diagram of $\mathcal{M}$ has
a single directed path from $\dpt$ to $\alt$, namely
$\dpt\overset 2\to\ang\overset 1\to\alt$, the induced
transitions through $M_1$ and $M_2$ must be $\dpt\overset
2\to\ang$ and $\ang\overset 1\to\alt$, respectively. It follows
that the two edges joining $M_1$ to $M_2$ receive from $\phi$
the values forming an angle $\{c_1+c_2,c_1+c_3\}$, where $c_1$,
$c_2$, and $c_3$ are corner points of $T$. Since the flow
through the input connector of $M_1$ is $0$, the Kirchhoff law
tells us that the residual edge of $M_1$ receives the value
$(c_1+c_2)+(c_1+c_3)=c_2+c_3$. The value on the residual edge
of $M_2$ must therefore be either $c_2$ or $c_3$, which in turn
implies that the flow through the output connector of $M_1\odot
M_2$ will be either $(c_2+c_3)+c_2=c_3$ or $(c_2+c_3)+c_3=c_2$.
In both cases, the outflow from $M_1\odot M_2$ has weight $1$,
while an altitude has trace of weight~$3$. This contradiction
excludes the transition $\dpt\overset 1\to\alt$. The reverse
transition $\alt\overset1\to\dpt$ can be handled similarly.
\end{proof}

For brevity we set
\begin{align}
\mathcal{M}'=\mathcal{M}\odot\mathcal{M}-\{\dpt\overset 1\oto\alt\}.
\end{align}
It is easy to check that $\mathcal{M}'\subseteq\mathcal{M}$,
which immediately implies the following result.

\begin{corollary}\label{cor:MM}
If $M_1$ and $M_2$ are $(2,2;1)$-poles such that
$\T(M_1)\subseteq\mathcal{M}$ and
$\T(M_2)\subseteq\mathcal{M}$, then $\T(M_1\odot
M_2)\subseteq\mathcal{M}'\subseteq\mathcal{M}$.
\end{corollary}

We divide the proof of the fact that $\pi(H^{\sharp})\ge 5$ for
every Halin snark into three smaller steps,
Theorems~\ref{thm:Halin(2,2;1)},~\ref{thm:Halin(2,2)}, and
\ref{thm:ext-Halin(2,2)}, each dealing with a particular dipole
contained in it. All three statements have an independent value
as they can be applied in the construction of graphs with
perfect matching index at least $5$.

\begin{definition}\label{def:Halin-dip}
{\rm Let $H^{\sharp}=C^{\sharp}\cup S$ be a Halin snark, with
$H=C\cup S$ being the underlying Halin graph, $C$ being the
perimeter circuit, and $S$ the inscribed tree. A \textit{Halin
$(2,2;1)$-pole} $X$ is a $(2,2,1)$-pole created from
$H^{\sharp}$ by removing one Halin fragment $F=B\circ D$ from
$C^{\sharp}$ and keeping the dangling edges. The connectors of
$X$ are formed in a natural manner: each consists of the
semiedges joined in $H^{\sharp}$ to a connector of $F$. A
\textit{Halin $(2,2)$-pole} $Y$ is a $(2,2)$-pole created from
$H^{\sharp}$ in a similar manner with the exception that one
does not remove the entire Halin fragment $F=D\circ B$, where
$D$ is a decollineator and $B$ is a bipartite $(2,2;1)$-pole,
but removes only $D$. Finally, an \textit{extended Halin
$(2,2)$-pole} $Z$ is a $(2,2)$-pole created from $H^{\sharp}$
by severing a pair of edges that connect two consecutive Halin
fragments in $C^{\sharp}$. In all three cases the output
connector coincides with the output connector of the Halin
fragment preceding $F$ with respect to the cyclic ordering
around~$C^{\sharp}$. }
\end{definition}

\begin{theorem}\label{thm:Halin(2,2;1)}
Every Halin $(2,2;1)$-pole $X$ satisfies
$\T(X)\subseteq\mathcal{M}'$.
\end{theorem}

\begin{proof}
Consider a Halin $(2,2;1)$-pole $X$ constructed from a Halin
snark $H^{\sharp}=C^{\sharp}\cup S$ by removing one Halin
fragment. Let $H=C\cup S$ be the underlying Halin graph.
Throughout the proof we refer to the notation introduced in the
beginning of this section.

We prove the result by induction on the number of trivalent
vertices of the inscribed tree $S$ of~$H$. If $S$ has only one
trivalent vertex, then, up to a cyclic shift of indices,
$X=F_0\odot F_1$ where $F_0$ and $F_1$ are Halin fragments.
Since $\T(F_0)\subseteq\mathcal{M}$ and
$\T(F_1)\subseteq\mathcal{M}$, Proposition~\ref{prop:M'}
implies that $\T(X)=\T(F_0\odot F_1)\subseteq\mathcal{M}'$.
This establishes the basis of induction.

For the induction step consider an arbitrary Halin
$(2,2;1)$-pole $X$ where the inscribed tree has $k\ge 2$
trivalent vertices, and assume that the result is true for all
Halin $(2,2;1)$-poles whose inscribed trees have fewer than $k$
trivalent vertices. The perimeter circuit of $H$ has $k+2$
vertices $v_0,v_1,\ldots, v_{k+1}$, which means that
$H^{\sharp}$ comprises $k+2$ Halin fragments $F_0, F_1, \ldots,
F_k, F_{k+1}$. Without loss of generality we may assume that
$X$ is constructed by removing $F_{k+1}$. Let $w$ be the vertex
of $X$ incident with its residual semiedge, and let $w_1$ and
$w_2$ be the neighbours of $w$ in $X$. For $i\in\{1,2\}$ let
$S_i$ denote the component of $S-ww_i$ containing the vertex
$w_i$, and let $S_i'$ be the tree obtained from $S_i$ by adding
to it the vertex $w$ and the edge $ww_i$. Recall that $H$ is
planar with $S$ being a plane tree. Therefore there exists an
index $r$ with $0\le r < k$ such that the vertices $v_0,
\ldots, v_r$ are the leaves of $S_1$ and the vertices $v_{r+1},
\ldots, v_{k}$ are the leaves of $S_2$. Now we can form two
smaller Halin graphs $H_1=C_1\cup S_1'$ and $H_2=C_2\cup S_2'$,
where $C_1$ is  the circuit that runs through the vertices
$v_0, \ldots, v_r$, and~$w$, while $C_2$ is the circuit that
runs through the vertices $v_{r+1}, \ldots, v_{k}$, and~$w$.
From these Halin graphs we can build Halin snarks
$H_1^{\sharp}$ and $H_2^{\sharp}$ where $w$ is substituted
(say) by $F_{k+1}$ and the remaining perimeter vertices are
substituted by the same Halin fragments as in~$H^{\sharp}$. By
removing $F_{k+1}$ from both $H_1^{\sharp}$ and  $H_2^{\sharp}$
we obtain Halin $(2,2;1)$-poles $X_1$ and $X_2$, respectively.
It is easy to see that $X_1\odot X_2$ is isomorphic to $X$. The
induction hypothesis implies that
$\T(X_1)\subseteq\mathcal{M}'$ and
$\T(X_2)\subseteq\mathcal{M}'$. Since
$\mathcal{M}'\subseteq\mathcal{M}$, from
Proposition~\ref{prop:M'} eventually get $\T(X)=\T(X_1\circ
X_2)\subseteq\mathcal{M}'$. This concludes the induction step
as well as the entire proof.
\end{proof}

\begin{theorem}\label{thm:Halin(2,2)}
Every transition through a Halin $(2,2)$-pole has the form
$$\ls\to\ls \quad {or } \quad \hl\to\hl.$$
\end{theorem}

\begin{proof}
Let $Y$ be an arbitrary Halin $(2,2)$-pole.
Definition~\ref{def:Halin-dip} implies that $Y$ can be obtained
from a bipartite $(2,2;1)$-pole $B$ and a Halin $(2,2;1)$-pole
$X$ by taking the join $B\circ X$ and then performing the
junction of the residual semiedges of $B$ and $X$. It follows
that a transition $\mathtt{s}\to\mathtt{t}$ through $Y$ can
only exist provided that there exists a transition
$\mathtt{s}\overset i\to\mathtt{q}$ through $B$ and a
transition $\mathtt{q}\overset j\to\mathtt{t}$ through $X$ such
that $i=j$. Furthermore, by Theorem~\ref{thm:admiss-trans}, the
resulting transition $\mathtt{s}\to\mathtt{t}$ must be
admissible, which means that it belongs to the set
$\mathcal{A}$ displayed in Equation (\ref{eq:A}).
Theorem~\ref{thm:wtrans-bip} and Theorem~\ref{thm:Halin(2,2;1)}
imply that $\T(B)\subseteq\mathcal{B}$ and
$\T(X)\subseteq\mathcal{M}'$. By inspecting the relations
$\mathcal{B}$ and $\mathcal{M}'$ and excluding all results that
are not admissible we conclude that only the following five
combinations $(\mathtt{s}\overset
i\to\mathtt{q})\in\mathcal{B}$ and $(\mathtt{q}\overset
i\to\mathtt{t})\in\mathcal{M}'$ are possible:
\begin{align}
&\hl\overset 1\to\ls\overset 1\to\hl 
\quad & \ls\overset 2\to\ang\overset 2\to\ls \nonumber\\
&\hl\overset 2\to\alt\overset 2\to\hl
\quad & \ls\overset 1\to\alt\overset 1\to\ls\nonumber\\
&\hl\overset 1\to\dpt\overset 1\to\hl\nonumber
\end{align}
In all the cases we obtain either $\ls\to\ls$ or $\hl\to\hl$.
The theorem is proved.
\end{proof}

We have just proved that all transitions through a Halin
$(2,2)$-pole are collinear. A Halin $(2,2)$-pole is thus an
example of a \textit{collineator}, a $(2,2)$-pole whose
transition relation is contained in the set
$\mathcal{C}=\{\ls\to\ls, \hl\to\hl\}$ of all collinear
transitions.

\begin{remark}{\rm
Both transitions from Proposition~\ref{thm:Halin(2,2)} can
actually be achieved by the Halin $(2,2)$-pole on 26 vertices
obtained from the windmill snark of order 34 depicted in
Figure~\ref{fig:g34+46}~(left) by deleting one decollineator on
$8$ vertices.} \qed
\end{remark}

\begin{theorem}\label{thm:ext-Halin(2,2)}
Every transition through an extended Halin $(2,2)$-pole has
the form $$\ang\to\ls.$$
\end{theorem}

\begin{proof}
Let $Z$ be an arbitrary extended Halin $(2,2)$-pole.
Definition~\ref{def:Halin-dip} implies that $Z=D\circ Y$ where
$D$ is a decollineator and  $Y$ is a Halin $(2,2)$-pole. By
using Corollary~\ref{cor:D} and Theorem~\ref{thm:Halin(2,2)} we
can conclude that a transition $\mathtt{s}\to\mathtt{t}$
through $Z$ can only exist when there exists a transition
$\mathtt{s}\to\mathtt{q}$ in $\mathcal{D}$, and
$\mathtt{q}\to\mathtt{t}$ is one of $\ls\to\ls$ and
$\hl\to\hl$. However, $\mathcal{D}$ has no transition involving
$\hl$, so $\mathtt{q}=\ls$ and hence $\mathtt{s}=\ang$ and
$\mathtt{t}=\ls$, as required.
\end{proof}

\begin{theorem}\label{thm:Halin}
Every Halin snark has perfect matching index at least $5$.
\end{theorem}

\begin{proof}
Every Halin snark $G$ can be expressed as $G=[Z]$ where $Z$ is
an extended Halin $(2,2)$-pole. Now, if we had $\pi(G)\le 4$,
then $G$ would have a tetrahedral flow, and hence the
$(2,2)$-pole $Z$ would admit a stationary transition. However,
Theorem~\ref{thm:ext-Halin(2,2)} shows that this is not the
case. Therefore $\pi(G)\ge 5$.
\end{proof}

The next theorem makes use of the fact that the family of Halin
snarks is quite rich.

\begin{theorem}~\label{thm:42}
For every even integer $n\ge 42$ there exists a nontrivial
snark of order $n$ with perfect matching index at least $5$.
\end{theorem}

\begin{proof}
We prove that there exists a nontrivial Halin snark of every
even order $n\ge 42$. We know that there exists a treelike
snark order $n=12k-2$ for every integer $k\ge 3$. These snarks
are clearly nontrivial and are Halin. It is therefore
sufficient to establish the existence of nontrivial Halin
snarks of order $n$ for the remaining five even residue classes
modulo 12. To this end, we prove the following two claims.

\medskip\noindent
Claim 1. \textit{If there exists a nontrivial Halin snark on
$n$ vertices, then there is also one on $n+12$ vertices.}

\medskip\noindent
Proof of Claim~1. Let $H^{\sharp}=C^{\sharp}\cup T$ be a
nontrivial Halin snark of order $n$. In $H^{\sharp}$, choose an
arbitrary Halin fragment $F_h=D_h\circ B_h\subseteq
C^{\sharp}$, where $D_h$ is a decollineator and $B_h$ is a
bipartite $(2,2;1)$-pole. Take the Heawood graph, denoted by
$Hw$, which is a bipartite vertex- and edge-transitive cubic
graph of girth $6$ on 14 vertices, and create the $(2,2)$-pole
$G_{uv}$ for any fixed edge $uv$ of $G=Hw$; denote it by
$M_{Hw}$. Now, construct a new graph $K$ by replacing in
$H^{\sharp}$ the Halin fragment $F_h=D_h\circ B_h$ with the
$(2,2;1)$-pole $D_h\circ B_h\circ M_{Hw}$. It is easy to see
that $B_h\circ M_{Hw}$ is again a bipartite $(2,2;1)$-pole, so
$D_h\circ (B_h\circ M_{Hw})$ is a Halin fragment and $K$ is a
Halin snark. Its order is $n+12$, and since $Hw$ is cyclically
$6$-edge-connected (see \cite[Theorem~17]{NS-cc}), $K$ is also
nontrivial. This completes the proof of Claim~1.

\medskip

To finish the proof it suffices to prove the following.

\medskip\noindent
Claim 2. \textit{There exist nontrivial Halin snarks of order
$42$, $44$, $48$, $50$, and $52$.}

\medskip\noindent
Proof of Claim~2. In order to construct a nontrivial Halin
snark of order 42 we start with the windmill snark $W_{34}$ of
order 34, shown in Figure~\ref{fig:g34+46} (left). It is
created from the unique Halin graph of order 4, the complete
graph $K_4$, by substituting each vertex of the perimeter
circuit with the Petersen fragment. Recall that the Petersen
fragment can be expressed as $F_{Ps}=D_{Ps}\circ B$ where
$D_{Ps}$ is the decollineator constructed from the Petersen
graph and $B$ is the bipartite $(2,2;1)$-pole $G_{uwv}$ created
from the complete bipartite graph $K_{3,3}$. Let us change
$K_{3,3}$ to the Heawood graph to obtain a bipartite
$(2,2;1)$-pole $B_{Hw}$ of order 11. Replacing in $W_{34}$ the
Petersen fragment with the Halin fragment $D_{Ps}\circ B_{Hw}$
gives rise to a Halin snark $G_{42}$ of order 42. Again,
$G_{42}$ is nontrivial because $Hw$ is cyclically
$6$-edge-connected.

Nontrivial Halin snarks of order 44, 48, 52, can be constructed
in a manner similar to $G_{42}$. Instead of the Heawood graph
one has to use, respectively, the generalised Petersen graph
$GP(8,3)$ (known as the M\"obius-Kantor graph), the generalised
Petersen graph $GP(10,3)$ (known as the Desargues graph), and
the generalised Petersen graph $GP(12,5)$ (nicknamed the Nauru
graph), all of them vertex- and edge-transitive cubic bipartite
graphs of girth 6 (see \cite{FC}). They are all cyclically
$6$-edge-connected due to \cite[Theorem~17]{NS-cc}, implying
that the resulting Halin snarks are nontrivial.

Finally, we construct a nontrivial Halin snark of order 50. For
this purpose we use the graph $G_{42}$ constructed above,
choose one Petersen fragment $F_{Ps}=D_{Ps}\circ B$ left intact
by the construction of $G_{42}$ from $W_{34}$, and replace it
with $D_{Ps}\circ B_{Hw}$. The number of vertices increases by
$8$ to $50$, and the result is obviously a nontrivial Halin
snark.

The proof is complete.
\end{proof}

\section{Circular flows on Halin snarks}\label{sec:Halin-cfn5}

In this section we show that Halin snarks provide a rich source
of graphs with circular flow number at least $5$. Such graphs
are particularly interesting for the outstanding $5$-flow
conjecture of Tutte and therefore have been extensively studied
in a number of recent papers \cite{AKLM, EMT, GMM, MR}.

Given a real number $r\ge 2$, we define a \textit{nowhere-zero
real-valued $r$-flow} as an $\mathbb{R}$-flow $\phi$ such that
$1\le |\phi(e)|\le r-1$   for each edge $e$ of $G$. A
\textit{nowhere-zero modular $r$-flow} is an
$\mathbb{R}/r\mathbb{Z}$-flow $\phi$ such that $1\le \phi(e)
\pmod r \le r-1$ for each edge $e$. Here $x \pmod r$ denotes
the unique real number $x'\in[0,r)\subseteq\mathbb{R}$ such
that $x-x'$ is a multiple of $r$. It is a well-known folklore
fact that a graph admits a nowhere-zero real-valued $r$-flow if
and only if it admits a nowhere-zero modular $r$-flow.

The \textit{circular flow number} of a graph $G$, denoted by
$\Phi_c(G)$, is the infimum of the set of all real numbers $r$
such that $G$ has a nowhere-zero $r$-flow. This parameter was
introduced by Goddyn et al. in \cite{Goddyn} as
\textit{fractional flow number} and was shown to be a minimum
and a rational number for every (finite) bridgeless graph.

\begin{theorem}\label{thm:cfn5}
Let $G$ be a Halin snark in which all decollineators have been
obtained from snarks with circular flow number at least $5$.
Then $\Phi_c(G)\ge 5$.
\end{theorem}

\begin{proof}
Let $H$ be a Halin graph with perimeter circuit $C=(v_0v_1\dots
v_{k-1})$ and inscribed tree $S$, and let $H^{\sharp}$ be a
Halin snark created from $H$ using decollineators obtained from
graphs with circular flow number at least 5. Recall that every
Halin graph contains a triangle; without loss of generality we
may assume that $v_0v_1u$ is a triangle in $H$, with $u$ being
a vertex of $S$. Set $e_0=v_0u$ and $e_1=v_1u$.

Suppose to the contrary that $\Phi_c(H^{\sharp})=r<5$, and let
$\phi$ be a nowhere-zero modular $r$-flow on $H^{\sharp}$.
Define the \textit{flow through a dipole $X$} as the sum of
flow-values (in $\mathbb{R}/r\mathbb{Z}$) on the dangling edges of the input 
connector directed towards the dipole; of course, this value coincides
with the sum of flow-values on the dangling edges in the output
connector of $X$ directed away from~$X$.

For $i\in\{0,1,\ldots,k-1\}$ let $a_i$ denote the flow through
the decollineator $D_i$ of $H^{\sharp}$. By our assumption,
each $D_i$ arises from a snark with circular flow number at
least 5, so each $a_i$ lies in the interval $(-1,1)$. Without
loss of generality we may assume that $a_1\in [0,1)$. As $\phi$
is an $r$-flow, both $\phi(e_0)$ and $\phi(e_1)$ are contained
in the interval $[1,r-1]$. It follows that both $a_0$ and $a_2$
are contained in $(-1,0]$. This in turn implies that
$\phi(e_0)\in[1,2)$ and $\phi(e_1)\in(-2,-1]$, provided that
$e_0$ and $e_1$ are in directed from $u$. However, the third
edge incident with $u$ now receives a value from the interval
$(-1,1)$, which contradicts the definition of a nowhere-zero
$r$-flow.
\end{proof}

The just proved theorem significantly generalises the result of
Abreu et al. \cite[Theorem~9]{AKLM} that the circular flow
number of every treelike snark is at least $5$. Indeed, every
treelike snark can serve as an ingredient for a Halin fragment that
satisfies the assumptions of Theorem~\ref{thm:cfn5} and thus
gives rises to new snarks with circular flow number at least $5$.
Clearly, this process can be iterated indefinitely.

\section{Concluding remarks}

The methods introduced in this paper can be used to produce a great variety of  
families of snarks with perfect matching index at least $5$. For example, Halin 
dipoles treated in Theorems~\ref{thm:Halin(2,2)} and~\ref{thm:ext-Halin(2,2)} 
can be variously combined with bipartite $(2;2;1)$-poles and dipoles 
constructed on the basis of Propositions~\ref{prop:bip22pole} 
and~\ref{prop:2decol}.
The family of Halin snarks itself
can be easily extended to a family of snarks of the form
$H^{\sharp}\cup S$, where $S$ is a
planar forest with any number of components rather than being just a
tree. One possibility to construct such a snark is to start 
with a Halin snark expressed as $G=[D\circ Y]$ where $D$ is a
decollineator and $Y$ is a Halin $(2,2)$-pole. Since $Y$ is a
collineator and the composition of two collineators is again a
collineator, we can create a graph $G^+_1=[D\circ
(Y_1\circ\ldots\circ Y_k)]$ where $Y_1,\ldots, Y_k$ are
arbitrary Halin $(2,2)$-poles. Corollary~\ref{cor:D} and
Theorem~\ref{thm:Halin(2,2)} immediately imply that $G^+_1$ has
no tetrahedral flow, whence $\pi(G^+_1)\ge 5$. Another possibility is
to take a graph $G^+_2=[Z_1\circ\ldots\circ Z_k]$ where
$Z_1,\ldots, Z_k$ are extended Halin $(2,2)$-poles. In this
case, the fact that $\pi(G^+_2)\ge 5$ follows from
Theorem~\ref{thm:ext-Halin(2,2)}. We can also develop special constructions 
such as superposition \cite{MS-pmi+superp} or a construction of nontrivial 
snarks with $\pi\ge 5$ and circular flow number strictly smaller than $5$, see 
\cite{MS-pmi+cfn}.

Our methods unfortunately fail for nontrivial snarks with perfect matching 
index at least $5$ of small order. We therefore leave the following open 
problem.

\begin{problem}
Do there exist nontrivial snarks of orders $38$ and $40$ with perfect matching 
index at least $5$?
\end{problem}

If $4$-cycles are permitted, then the answer to this question is positive. 
Cyclically $4$-edge-connected snarks with $\pi\ge 5$ containing a quadrilateral 
can be easily created from the snark shown in 
Figure~\ref{fig:g34+46}~(left) by replacing one or two copies of the bipartite 
$(2,2;1)$-pole on three vertices obtained from the complete bipartite graph 
$K_{3,3}$ with a similar $(2,2;1)$-pole created from the graph 
of the $3$-dimensional cube.

\subsection*{Acknowledgements}
The authors acknowledge partial support from the grants VEGA
1/0813/18 and APVV-15-0220.

%%%%%%%%%%%%%%%%%%%%%%%%%%%%%%%%%%%%%%%%%%%%%%%%%%%%%%%%%
%                    REFERENCES                         %
%%%%%%%%%%%%%%%%%%%%%%%%%%%%%%%%%%%%%%%%%%%%%%%%%%%%%%%%%

\end{document}

%% file: pp.tex
\begin{picture}(0,0)%
\includegraphics{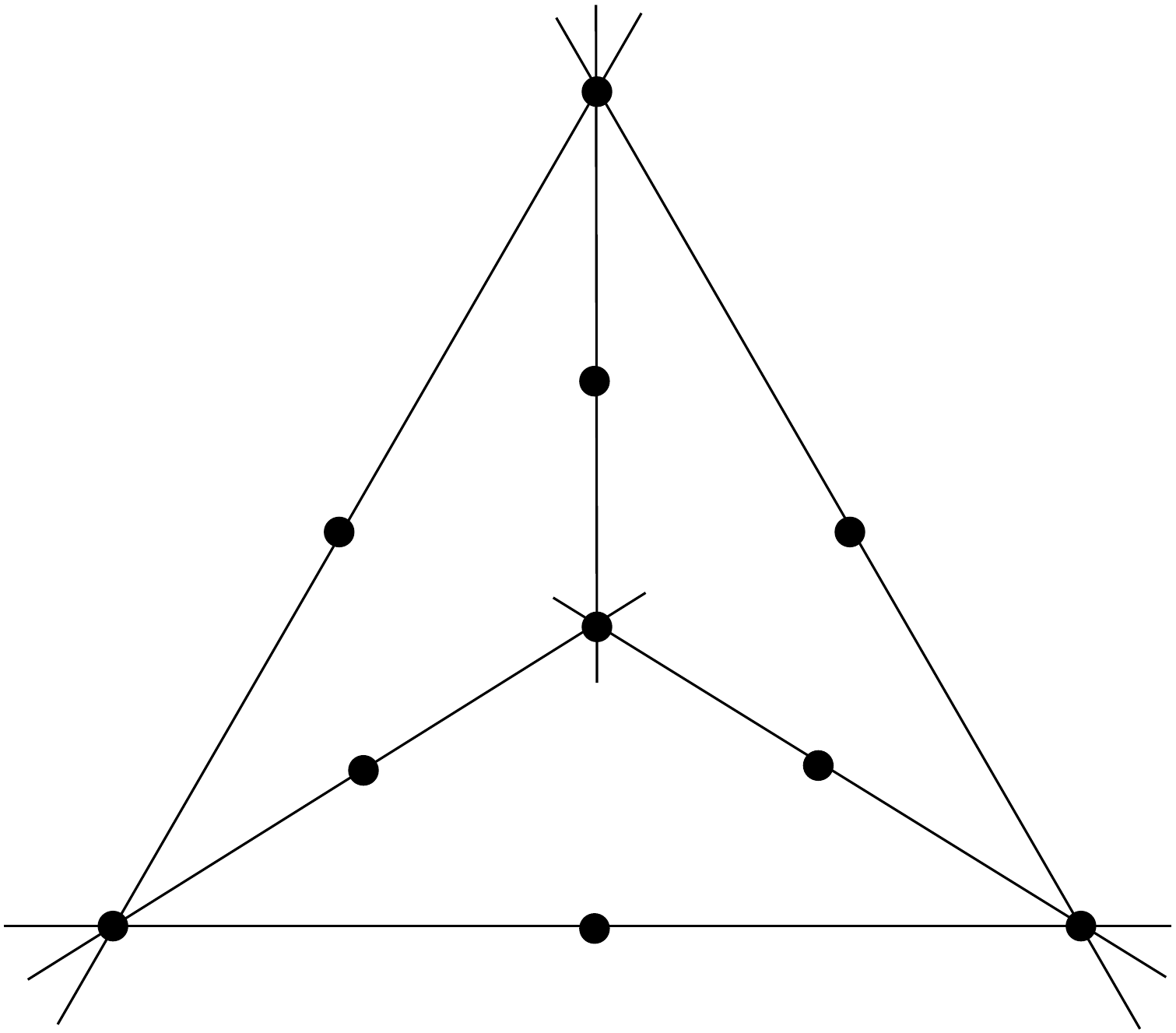}%
\end{picture}%
\setlength{\unitlength}{3947sp}%
\begingroup\makeatletter\ifx\SetFigFont\undefined%
\gdef\SetFigFont#1#2#3#4#5{%
  \reset@font\fontsize{#1}{#2pt}%
  \fontfamily{#3}\fontseries{#4}\fontshape{#5}%
  \selectfont}%
\fi\endgroup%
\begin{picture}(7244,6358)(2979,-16619)
\put(6046,-16383){\makebox(0,0)[lb]{\smash{{\SetFigFont{25}{30.0}{\rmdefault}{\mddefault}{\updefault}{\color[rgb]{0,0,0}$p_1+p_2$}%
}}}}
\put(6916,-10953){\makebox(0,0)[lb]{\smash{{\SetFigFont{25}{30.0}{\rmdefault}{\mddefault}{\updefault}{\color[rgb]{0,0,0}$p_3$}%
}}}}
\put(8446,-13593){\makebox(0,0)[lb]{\smash{{\SetFigFont{25}{30.0}{\rmdefault}{\mddefault}{\updefault}{\color[rgb]{0,0,0}$p_2+p_3$}%
}}}}
\put(9781,-15708){\makebox(0,0)[lb]{\smash{{\SetFigFont{25}{30.0}{\rmdefault}{\mddefault}{\updefault}{\color[rgb]{0,0,0}$p_2$}%
}}}}
\put(3346,-15723){\makebox(0,0)[lb]{\smash{{\SetFigFont{25}{30.0}{\rmdefault}{\mddefault}{\updefault}{\color[rgb]{0,0,0}$p_1$}%
}}}}
\put(7051,-15333){\makebox(0,0)[lb]{\smash{{\SetFigFont{25}{30.0}{\rmdefault}{\mddefault}{\updefault}{\color[rgb]{0,0,0}$p_2+p_4$}%
}}}}
\put(6991,-14193){\makebox(0,0)[lb]{\smash{{\SetFigFont{25}{30.0}{\rmdefault}{\mddefault}{\updefault}{\color[rgb]{0,0,0}$p_4$}%
}}}}
\put(5251,-15333){\makebox(0,0)[lb]{\smash{{\SetFigFont{25}{30.0}{\rmdefault}{\mddefault}{\updefault}{\color[rgb]{0,0,0}$p_1+p_4$}%
}}}}
\put(3661,-13578){\makebox(0,0)[lb]{\smash{{\SetFigFont{25}{30.0}{\rmdefault}{\mddefault}{\updefault}{\color[rgb]{0,0,0}$p_1+p_3$}%
}}}}
\put(6151,-13053){\makebox(0,0)[lb]{\smash{{\SetFigFont{25}{30.0}{\rmdefault}{\mddefault}{\updefault}{\color[rgb]{0,0,0}$p_3+p_4$}%
}}}}
\end{picture}%

%% file: Halin.tex
\begin{picture}(0,0)%
\includegraphics{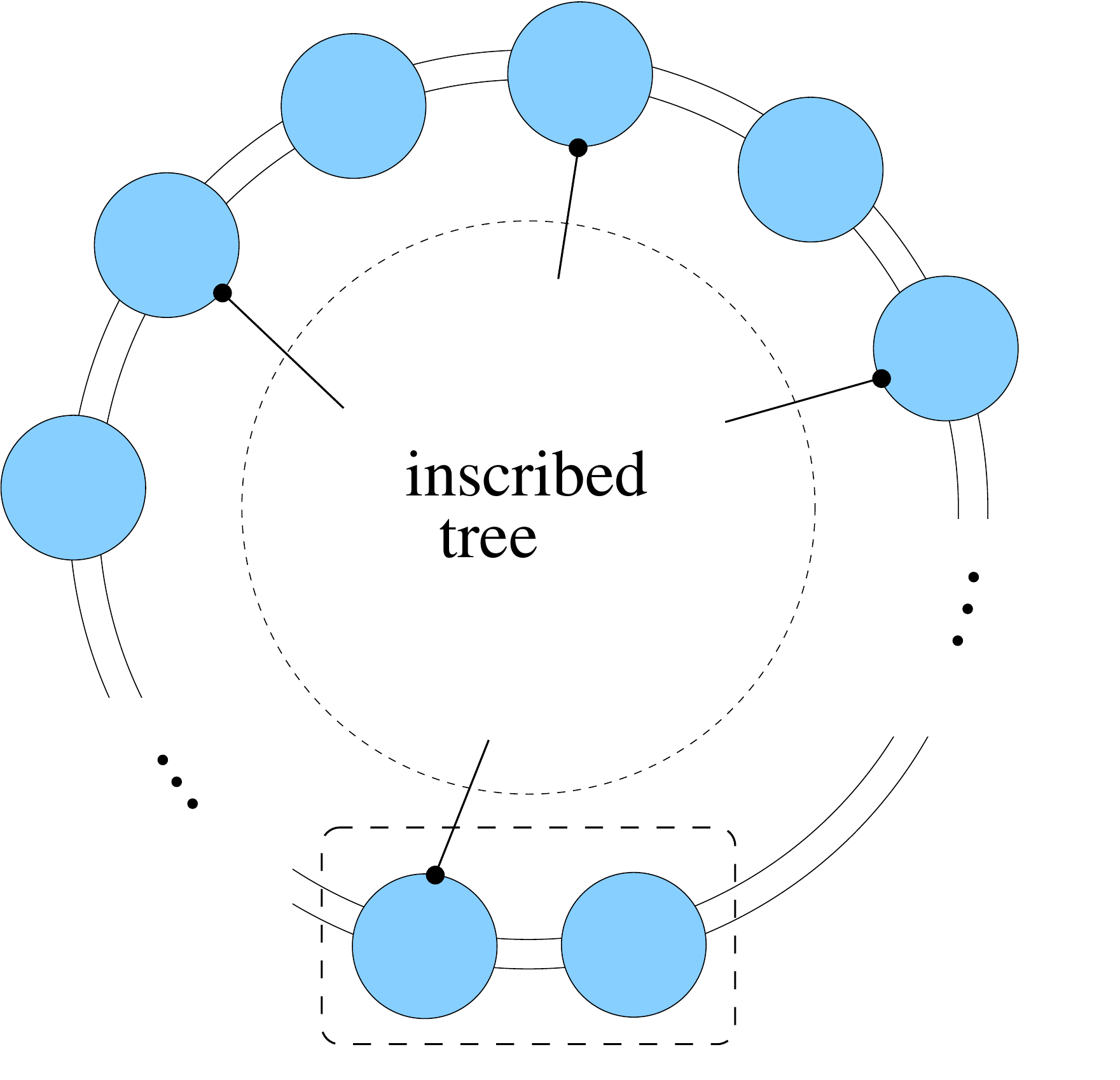}%
\end{picture}%
\setlength{\unitlength}{3947sp}%
\begingroup\makeatletter\ifx\SetFigFont\undefined%
\gdef\SetFigFont#1#2#3#4#5{%
  \reset@font\fontsize{#1}{#2pt}%
  \fontfamily{#3}\fontseries{#4}\fontshape{#5}%
  \selectfont}%
\fi\endgroup%
\begin{picture}(8441,8143)(2012,-7366)
\put(4966,-6496){\makebox(0,0)[lb]{\smash{{\SetFigFont{29}{34.8}{\rmdefault}{\mddefault}{\updefault}{\color[rgb]{0,0,0}$B_i$}%
}}}}
\put(3886,-1516){\makebox(0,0)[lb]{\smash{{\SetFigFont{29}{34.8}{\rmdefault}{\mddefault}{\updefault}{\color[rgb]{0,0,0}$v_k$}%
}}}}
\put(4486,-136){\makebox(0,0)[lb]{\smash{{\SetFigFont{29}{34.8}{\rmdefault}{\mddefault}{\updefault}{\color[rgb]{0,0,0}$D_0$}%
}}}}
\put(6556,-676){\makebox(0,0)[lb]{\smash{{\SetFigFont{29}{34.8}{\rmdefault}{\mddefault}{\updefault}{\color[rgb]{0,0,0}$v_0$}%
}}}}
\put(8506,-2611){\makebox(0,0)[lb]{\smash{{\SetFigFont{29}{34.8}{\rmdefault}{\mddefault}{\updefault}{\color[rgb]{0,0,0}$v_1$}%
}}}}
\put(6256,-3391){\makebox(0,0)[lb]{\smash{{\SetFigFont{29}{34.8}{\rmdefault}{\mddefault}{\updefault}{\color[rgb]{0,0,0}$S$}%
}}}}
\put(6181,104){\makebox(0,0)[lb]{\smash{{\SetFigFont{29}{34.8}{\rmdefault}{\mddefault}{\updefault}{\color[rgb]{0,0,0}$B_0$}%
}}}}
\put(7921,-631){\makebox(0,0)[lb]{\smash{{\SetFigFont{29}{34.8}{\rmdefault}{\mddefault}{\updefault}{\color[rgb]{0,0,0}$D_1$}%
}}}}
\put(8926,-1981){\makebox(0,0)[lb]{\smash{{\SetFigFont{29}{34.8}{\rmdefault}{\mddefault}{\updefault}{\color[rgb]{0,0,0}$B_1$}%
}}}}
\put(2338,-3046){\makebox(0,0)[lb]{\smash{{\SetFigFont{29}{34.8}{\rmdefault}{\mddefault}{\updefault}{\color[rgb]{0,0,0}$D_k$}%
}}}}
\put(3033,-1152){\makebox(0,0)[lb]{\smash{{\SetFigFont{29}{34.8}{\rmdefault}{\mddefault}{\updefault}{\color[rgb]{0,0,0}$B_k$}%
}}}}
\put(7711,-7216){\makebox(0,0)[lb]{\smash{{\SetFigFont{29}{34.8}{\rmdefault}{\mddefault}{\updefault}{\color[rgb]{0,0,0}$F_i$}%
}}}}
\put(6541,-6511){\makebox(0,0)[lb]{\smash{{\SetFigFont{29}{34.8}{\rmdefault}{\mddefault}{\updefault}{\color[rgb]{0,0,0}$D_i$}%
}}}}
\put(5491,-5806){\makebox(0,0)[lb]{\smash{{\SetFigFont{29}{34.8}{\rmdefault}{\mddefault}{\updefault}{\color[rgb]{0,0,0}$v_i$}%
}}}}
\end{picture}%

%% file: prechodyB.tex
\begin{picture}(0,0)%
\includegraphics{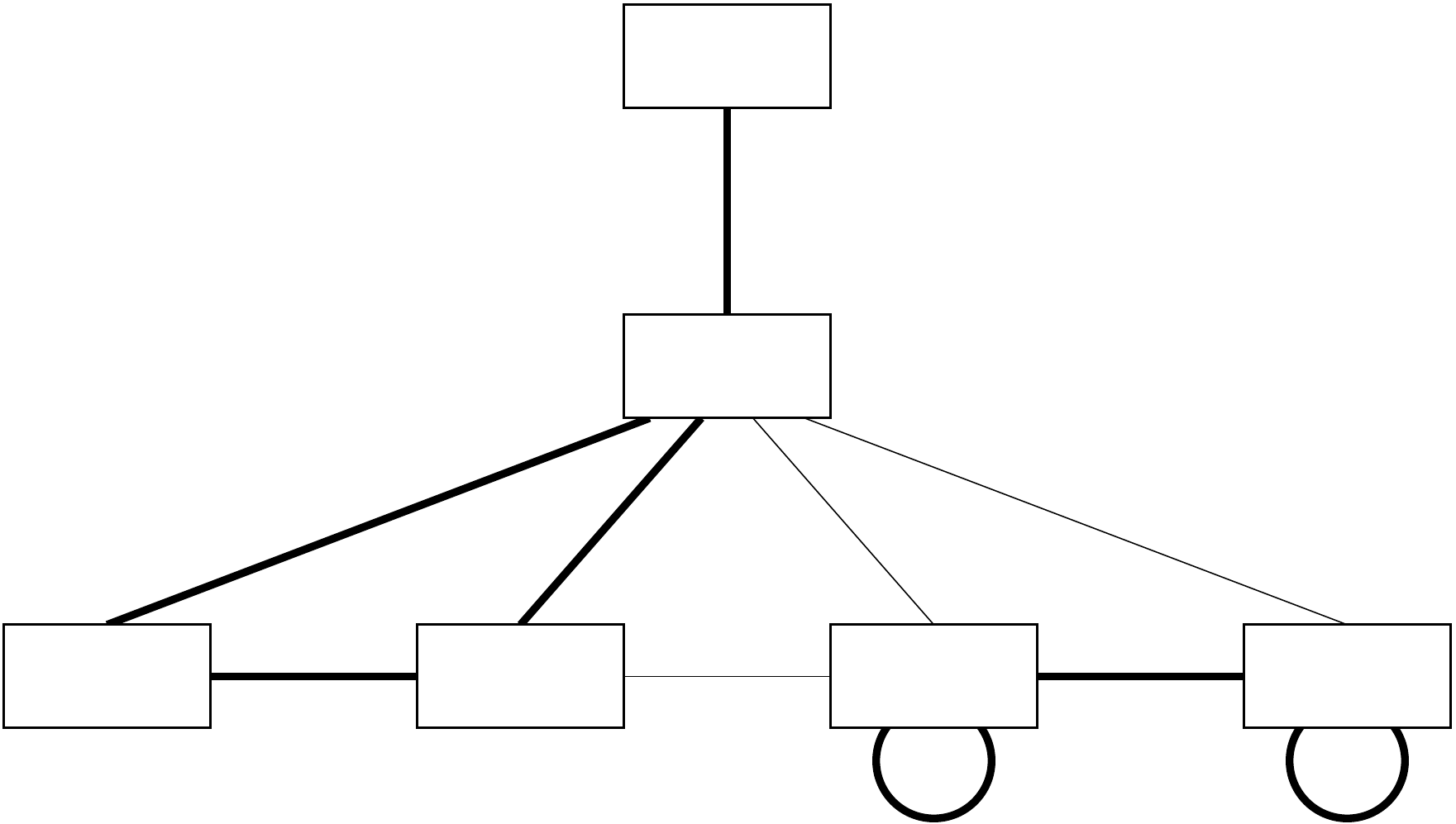}%
\end{picture}%
\setlength{\unitlength}{3947sp}%
\begingroup\makeatletter\ifx\SetFigFont\undefined%
\gdef\SetFigFont#1#2#3#4#5{%
  \reset@font\fontsize{#1}{#2pt}%
  \fontfamily{#3}\fontseries{#4}\fontshape{#5}%
  \selectfont}%
\fi\endgroup%
\begin{picture}(8444,4777)(1779,-5116)
\put(9301,-4336){\makebox(0,0)[lb]{\smash{{\SetFigFont{20}{24.0}{\rmdefault}{\mddefault}{\updefault}{\color[rgb]{0,0,0}$\mathtt{alt}$}%
}}}}
\put(5851,-2536){\makebox(0,0)[lb]{\smash{{\SetFigFont{20}{24.0}{\rmdefault}{\mddefault}{\updefault}{\color[rgb]{0,0,0}$\mathtt{ls}$}%
}}}}
\put(2101,-4336){\makebox(0,0)[lb]{\smash{{\SetFigFont{20}{24.0}{\rmdefault}{\mddefault}{\updefault}{\color[rgb]{0,0,0}$\mathtt{ang}$}%
}}}}
\put(4501,-4336){\makebox(0,0)[lb]{\smash{{\SetFigFont{20}{24.0}{\rmdefault}{\mddefault}{\updefault}{\color[rgb]{0,0,0}$\mathtt{dpt}$}%
}}}}
\put(6976,-4336){\makebox(0,0)[lb]{\smash{{\SetFigFont{20}{24.0}{\rmdefault}{\mddefault}{\updefault}{\color[rgb]{0,0,0}$\mathtt{hl}$}%
}}}}
\put(5776,-736){\makebox(0,0)[lb]{\smash{{\SetFigFont{20}{24.0}{\rmdefault}{\mddefault}{\updefault}{\color[rgb]{0,0,0}$\mathtt{ax}$}%
}}}}
\end{picture}%

%% file: prechodyDB.tex
\begin{picture}(0,0)%
\includegraphics{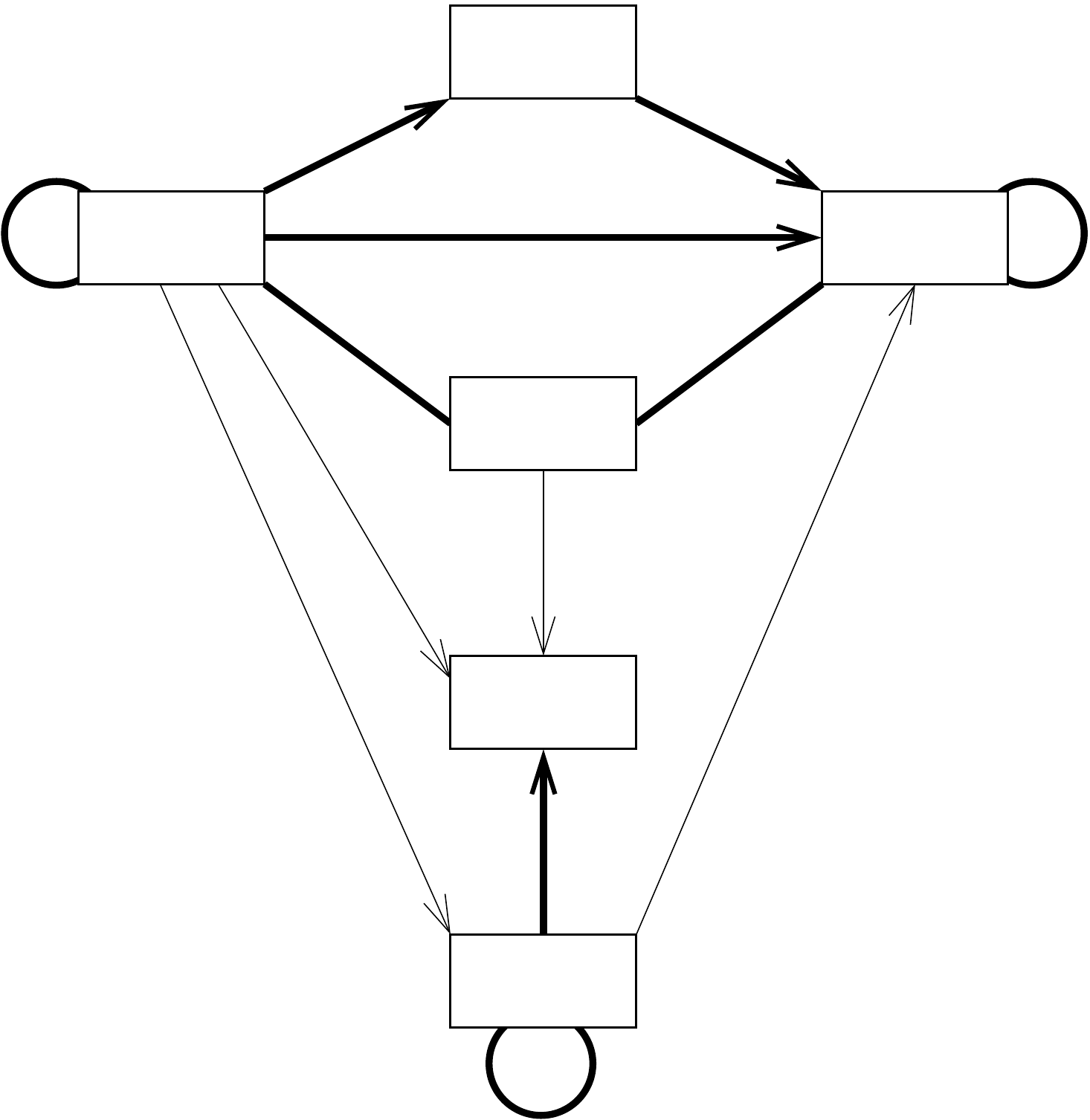}%
\end{picture}%
\setlength{\unitlength}{3947sp}%
\begingroup\makeatletter\ifx\SetFigFont\undefined%
\gdef\SetFigFont#1#2#3#4#5{%
  \reset@font\fontsize{#1}{#2pt}%
  \fontfamily{#3}\fontseries{#4}\fontshape{#5}%
  \selectfont}%
\fi\endgroup%
\begin{picture}(7030,7221)(2492,-12960)
\put(5701,-12136){\makebox(0,0)[lb]{\smash{{\SetFigFont{20}{24.0}{\rmdefault}{\mddefault}{\updefault}{\color[rgb]{0,0,0}$\mathtt{alt}$}%
}}}}
\put(5701,-8536){\makebox(0,0)[lb]{\smash{{\SetFigFont{20}{24.0}{\rmdefault}{\mddefault}{\updefault}{\color[rgb]{0,0,0}$\mathtt{dpt}$}%
}}}}
\put(5776,-6136){\makebox(0,0)[lb]{\smash{{\SetFigFont{20}{24.0}{\rmdefault}{\mddefault}{\updefault}{\color[rgb]{0,0,0}$\mathtt{ax}$}%
}}}}
\put(8176,-7336){\makebox(0,0)[lb]{\smash{{\SetFigFont{20}{24.0}{\rmdefault}{\mddefault}{\updefault}{\color[rgb]{0,0,0}$\mathtt{ls}$}%
}}}}
\put(5776,-10336){\makebox(0,0)[lb]{\smash{{\SetFigFont{20}{24.0}{\rmdefault}{\mddefault}{\updefault}{\color[rgb]{0,0,0}$\mathtt{hl}$}%
}}}}
\put(3301,-7336){\makebox(0,0)[lb]{\smash{{\SetFigFont{20}{24.0}{\rmdefault}{\mddefault}{\updefault}{\color[rgb]{0,0,0}$\mathtt{ang}$}%
}}}}
\end{picture}%

%% file: prechodyM.tex
\begin{picture}(0,0)%
\includegraphics{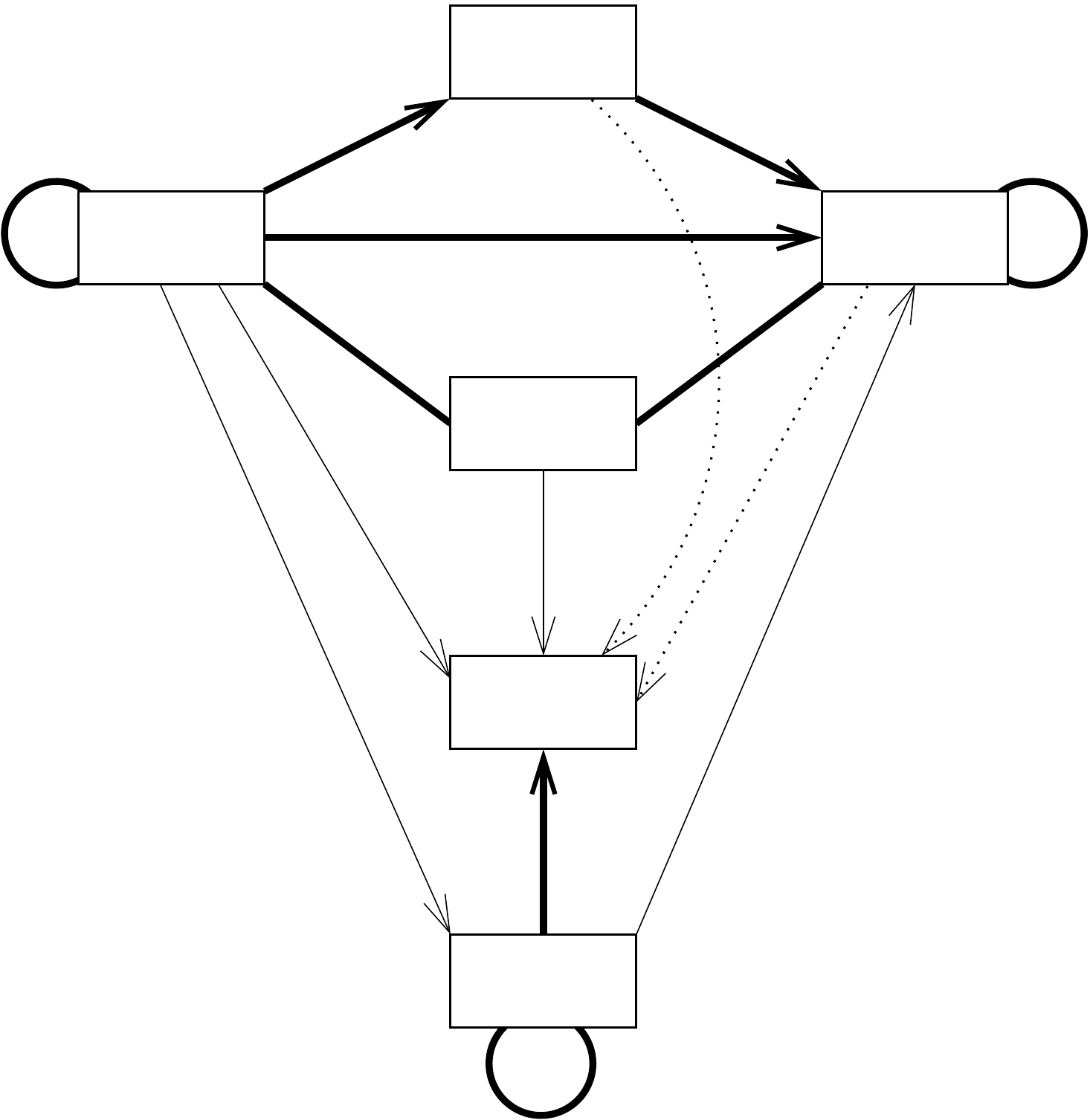}%
\end{picture}%
\setlength{\unitlength}{3947sp}%
\begingroup\makeatletter\ifx\SetFigFont\undefined%
\gdef\SetFigFont#1#2#3#4#5{%
  \reset@font\fontsize{#1}{#2pt}%
  \fontfamily{#3}\fontseries{#4}\fontshape{#5}%
  \selectfont}%
\fi\endgroup%
\begin{picture}(7030,7221)(2492,-12960)
\put(5701,-12136){\makebox(0,0)[lb]{\smash{{\SetFigFont{20}{24.0}{\rmdefault}{\mddefault}{\updefault}{\color[rgb]{0,0,0}$\mathtt{alt}$}%
}}}}
\put(5701,-8536){\makebox(0,0)[lb]{\smash{{\SetFigFont{20}{24.0}{\rmdefault}{\mddefault}{\updefault}{\color[rgb]{0,0,0}$\mathtt{dpt}$}%
}}}}
\put(5776,-6136){\makebox(0,0)[lb]{\smash{{\SetFigFont{20}{24.0}{\rmdefault}{\mddefault}{\updefault}{\color[rgb]{0,0,0}$\mathtt{ax}$}%
}}}}
\put(8176,-7336){\makebox(0,0)[lb]{\smash{{\SetFigFont{20}{24.0}{\rmdefault}{\mddefault}{\updefault}{\color[rgb]{0,0,0}$\mathtt{ls}$}%
}}}}
\put(5776,-10336){\makebox(0,0)[lb]{\smash{{\SetFigFont{20}{24.0}{\rmdefault}{\mddefault}{\updefault}{\color[rgb]{0,0,0}$\mathtt{hl}$}%
}}}}
\put(3301,-7336){\makebox(0,0)[lb]{\smash{{\SetFigFont{20}{24.0}{\rmdefault}{\mddefault}{\updefault}{\color[rgb]{0,0,0}$\mathtt{ang}$}%
}}}}
\end{picture}%

%% file: prechodyMM.tex
\begin{picture}(0,0)%
\includegraphics{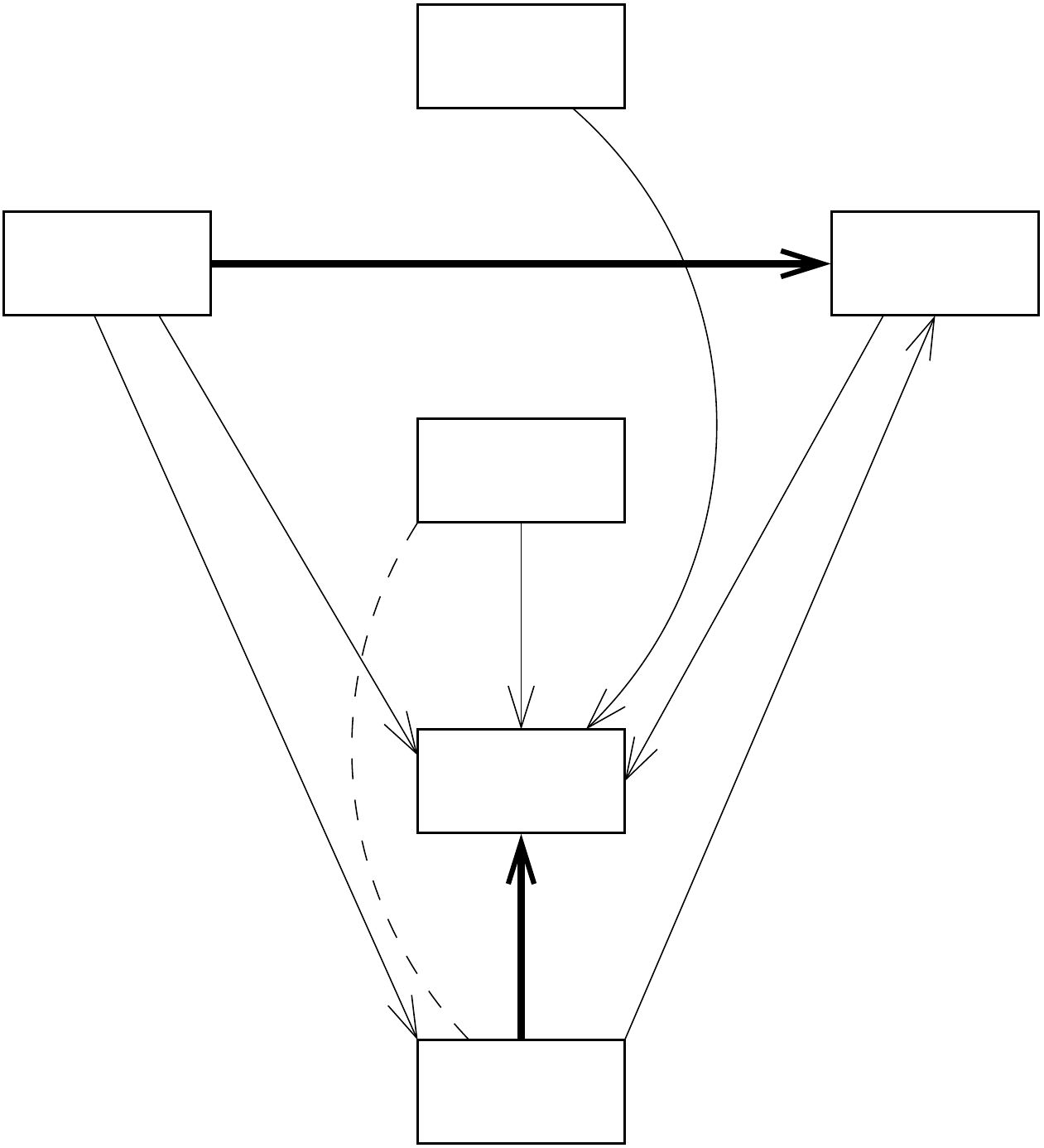}%
\end{picture}%
\setlength{\unitlength}{3947sp}%
\begingroup\makeatletter\ifx\SetFigFont\undefined%
\gdef\SetFigFont#1#2#3#4#5{%
  \reset@font\fontsize{#1}{#2pt}%
  \fontfamily{#3}\fontseries{#4}\fontshape{#5}%
  \selectfont}%
\fi\endgroup%
\begin{picture}(6044,6644)(2979,-12383)
\put(5701,-12136){\makebox(0,0)[lb]{\smash{{\SetFigFont{20}{24.0}{\rmdefault}{\mddefault}{\updefault}{\color[rgb]{0,0,0}$\mathtt{alt}$}%
}}}}
\put(5701,-8536){\makebox(0,0)[lb]{\smash{{\SetFigFont{20}{24.0}{\rmdefault}{\mddefault}{\updefault}{\color[rgb]{0,0,0}$\mathtt{dpt}$}%
}}}}
\put(5776,-6136){\makebox(0,0)[lb]{\smash{{\SetFigFont{20}{24.0}{\rmdefault}{\mddefault}{\updefault}{\color[rgb]{0,0,0}$\mathtt{ax}$}%
}}}}
\put(8176,-7336){\makebox(0,0)[lb]{\smash{{\SetFigFont{20}{24.0}{\rmdefault}{\mddefault}{\updefault}{\color[rgb]{0,0,0}$\mathtt{ls}$}%
}}}}
\put(5776,-10336){\makebox(0,0)[lb]{\smash{{\SetFigFont{20}{24.0}{\rmdefault}{\mddefault}{\updefault}{\color[rgb]{0,0,0}$\mathtt{hl}$}%
}}}}
\put(3301,-7336){\makebox(0,0)[lb]{\smash{{\SetFigFont{20}{24.0}{\rmdefault}{\mddefault}{\updefault}{\color[rgb]{0,0,0}$\mathtt{ang}$}%
}}}}
\end{picture}%

%% file: tetra7d.bbl
\begin{thebibliography}{MM}\frenchspacing

\bibitem{AKLM} M. Abreu, T. Kaiser, D. Labbate, G. Mazzuoccolo,
    Treelike snarks, Electron. J. Combin. 23 (2016), $\#$P3.54

% \bibitem{AT} N. Alon and M. Tarsi, Covering
%     multigraphs by simple circuits, SIAM J. Algebraic Discrete
%     Methods \textbf{6} (1985), 345--350.

\bibitem{A} D. Archdeacon, Coverings of graphs by cycles,
    Congr. Numer. 53 (1986), 7--14.

\bibitem{BGHM} G. Brinkmann, J. Goedgebeur, J. H\"agglund,
    K. Markstr\"om, Generation and properties of snarks,
    J. Combin. Theory Ser. B 103 (2013), 468--488.

\bibitem{Chen} F.~Chen, A note on Fouquet-Vanherpe's question
    and Fulkerson conjecture, Bull. Iranian Math. Soc. 42 (2016),
    1247--1258.

\bibitem{ChS} M. Chladn\'y, M. \v Skoviera, Factorisation of
    snarks, Electron. J. Combin. 17 (2010), \#R32.

\bibitem{Coxeter} H.~S.~M.~Coxeter, Projective Geometry,
    Springer-Verlag, New York, 1987.

\bibitem{EM} L. Esperet, G. Mazzuoccolo, On cubic bridgeless
    graphs whose edge-set cannot be covered by four perfect
    matchings, J. Graph Theory 77 (2014), 144--157.

\bibitem{EMT} L. Esperet, G. Mazzuoccolo, M. Tarsi, The
    structure of graphs with circular flow number $5$ or more,
    and the complexity of their recognition problem, J. Comb. 7
    (2016), 453--479.

%\bibitem{FR} G. Fan, A. Raspaud, Fulkerson's
%    conjecture and circuit covers, J.~Combin. Theory Ser. B
%    61 (1994), 133--138.

\bibitem{FR} G. Fan, A. Raspaud, Fulkerson's Conjecture
and circuit covers, J.~Combin. Theory Ser. B 61 (1994),
133--138.

\bibitem{FC} R. M. Foster, The Foster Census (edited by I. Z.
    Bouwer, W. W. Chernoff, B. Monson, and Z. Star), Charles
    Babbage Research Centre, Winnipeg, 1988. Extended online
    version:  G. Royle, M. Conder, B. McKay, P. Dobcs\'anyi,
    Cubic symmetric graphs (The Foster Census),
    \url{https://staffhome.ecm.uwa.edu.au/~00013890/remote/foster/}

\bibitem{FH} J. L. Fouquet, J. M. Vanherpe, On the perfect
    matching index of a bridgeless cubic graph,
    arXiv:0904.1296, 2009.

\bibitem{Goddyn} L. A. Goddyn, M. Tarsi, C. Zhang,
    On $(k,d)$-colorings and fractional nowhere-zero
    flows, J. Graph Theory \textbf{28} (1998), 155--161.

\bibitem{GMM} J. Goedgebeur, D. Mattiolo, G. Mazzuoccolo, A
    unifed approach to construct snarks with circular flow
    number $5$, arXiv:1804.00957.

\bibitem{Hagglund} J. H\"agglund, On snarks that are far from
    being $3$-edge colorable, Electron. J. Combin. 23 (2016),
    $\#$P2.6.

\bibitem{Halin} R. Halin, Studies on minimally $n$-connected
    graphs, in: Combinatorial Mathematics and its Applications,
    D. J. A. Welsh (ed.), Academic Press, New York, 1971, pp.
    129--136.

\bibitem{HS} F. Holroyd, M. \v Skoviera, Colouring of cubic
    graphs by Steiner triple systems, J. Combin. Theory Ser.~B
    \textbf{91} (2004), 57--66.

\bibitem{HLZ} X. Hou, H.-J. Lai, C.-Q. Zhang, On perfect
    matching coverings and even subgraph coverings, J. Graph
    Theory 81 (2015), 83--91.

\bibitem{Jaeger} F. Jaeger, Nowhere-zero flow problems, in: L.
    W. Beineke, R. J. Wilson (Eds.), Selected Topics in Graph
    Theory Vol. 3, Academic Press, London, 1988, pp. 71--95.

\bibitem{Kirkman} T. P. Kirkman, On the enumeration of $x$-edra
    having triedral summits, and an $(x-1)$-gonal base, Philos.
    Trans. Roy. Soc. 146 (1856), 399--411.

\bibitem{KMS} D. Kr\'al', E. M\'a\v cajov\'a, O. Pangr\'ac, A.
    Raspaud, J.-S. Sereni, M. \v Skoviera, Projective, affine,
    and abelian colorings of cubic graphs, European J.
    Combin. 30 (2009), 53--69.

\bibitem{Konig} D. K\"onig, Theorie der endlichen und
    unendlichen Graphen, Akademische Verlagsgesellschaft,
    Leipzig, 1936. Reprint: AMS Chelsea Publihing, 2001.

\bibitem{LP} L. Lov\'asz, M. D. Plummer, Matching Theory, AMS
    Chelsea Publishing Series Vol.~367, North-Holland;
    Elsevier Science Publishers B.V., 2009

\bibitem{MR} E. M\'a\v cajov\'a and A. Raspaud,
    On the Strong Circular $5$-Flow Conjecture, J.
    Graph Theory 52 (2006), 307--316.

\bibitem{MS-Fano} E. M\'a\v cajov\'a, M. \v Skoviera, Fano
    colourings of cubic graphs and the Fulkerson Conjecture,
    Theoret. Comput. Sci. 349 (2005), 112--120.
    
\bibitem{MS-pmi+cfn} E. M\'a\v cajov\'a, M. \v Skoviera, Perfect matching index 
vs. circular flow number of a cubic graph, submitted.

\bibitem{MS-pmi+superp} E. M\'a\v cajov\'a, M. \v Skoviera, Superposition of 
graphs with perfect matching index at least $5$, in progress.


\bibitem{M1} G. Mazzuoccolo, The equivalence of two conjectures
    of Berge and Fulkerson, J.~Graph Theory 68 (2011), 125–-128.

\bibitem{NS-cc} R.~Nedela, M.~\v Skoviera, Atoms of cyclic
    connectivity in cubic graphs, Math. Slovaca 45
    (1995), 481--499.

\bibitem{Pet} J. Petersen, Die Theorie der
    regul\"aren Graphs, Acta Math. 15 (1891), 193--220.

\bibitem{Plesnik} J. Plesn\'\i k, Connectivity of regular graphs
    and the existence of $1$-factors, Mat. \v casopis 22
    (1972), 310--318.

\bibitem{Rademacher}  H. Rademacher, On the number of certain
    types of polyhedra, Illinois J. Math. 9 (1965), 361--380.

\bibitem{Sch} T. Sch\"onberger, Ein Beweis des Petersenschen
    Graphensatzes, Acta Litt. Sci. Szeged 7 (1934), 51--57.

\bibitem{S} P. D. Seymour, On multi-colourings of cubic
    graphs, and conjectures of Fulkerson and Tutte, Proc.
    London Math. Soc. 38 (1979), 423--460.

\bibitem{Steffen} E. Steffen, 1-Factor and cycle covers of
    cubic graphs, J. Graph Theory 78 (2015), 195--206.

\end{thebibliography}
